\documentclass[11pt]{article}

\usepackage{latexsym}
\usepackage{amssymb}
\usepackage{amsthm}
\usepackage{amscd}
\usepackage{amsmath}
\usepackage{tikz}
\usepackage{mathabx}
\usepackage{stmaryrd}

\usepgflibrary{arrows}

\newtheorem{theorem}{Theorem}[section]

\newtheorem{lemma}[theorem]{Lemma}

\newtheorem{proposition}[theorem]{Proposition}
\newtheorem{corollary}[theorem]{Corollary}

\theoremstyle{definition}
\newtheorem*{example}{Example}
\newtheorem*{examples}{Examples}
\newtheorem*{remark}{Remark}
\newtheorem*{remarks}{Remarks}

\numberwithin{equation}{section}

\newcommand{\CC}{\mathbb{C}} 
\newcommand{\FF}{\mathbb{F}}

\newcommand{\ZZ}{\mathbb{Z}}
\newcommand{\RR}{\mathbb{R}}

\newcommand{\cA}{\mathcal{A}}
\newcommand{\cB}{\mathcal{B}} 
\newcommand{\cC}{\mathcal{C}}
\newcommand{\cD}{\mathcal{D}}

\newcommand{\cK}{\mathcal{K}}
\newcommand{\cL}{\mathcal{L}}

\newcommand{\cN}{\mathcal{N}}

\newcommand{\cP}{\mathcal{P}}
\newcommand{\cQ}{\mathcal{Q}}
\newcommand{\cR}{\mathcal{R}}
\newcommand{\cS}{\mathcal{S}}
\newcommand{\cT}{\mathcal{T}}

\newcommand{\cX}{\mathcal{X}}

\newcommand{\fkn}{\mathfrak{n}}

\newcommand{\fkut}{\mathfrak{ut}}

\newcommand{\Hom}{\mathrm{Hom}}

\newcommand{\GL}{\mathrm{GL}}
  
\newcommand{\Res}{\mathrm{Res}}

\newcommand{\Irr}{\mathrm{Irr}}

\newcommand{\wt}{\mathrm{wt}}

\newcommand{\Id}{\mathrm{Id}}

\newcommand{\UT}{\mathrm{UT}}

\newcommand{\nst}{\mathrm{nst}}

\newcommand{\Ext}{\mathrm{Ext}}

\newcommand{\dd}{\displaystyle}
\newcommand{\scs}{\scriptstyle}
\newcommand{\scscs}{\scriptscriptstyle}

\def\adots{\mathinner{\mkern2mu\raise0pt\hbox{.}  
\mkern2mu\raise4pt\hbox{.}\mkern1mu
\raise7pt\vbox{\kern7pt\hbox{.}}\mkern1mu}}

\newcommand{\crs}{\mathrm{crs}}

\newcommand{\bdry}{\mathrm{bdry}}
\newcommand{\rank}{\mathrm{rank}}

\newcommand{\levi}{\mathrm{lc}}
\newcommand{\fat}{\mathrm{fat}}

\newcommand{\parb}{\mathrm{pb}}
\newcommand{\qbin}[2]{\genfrac{[}{]}{0pt}{}{#1}{#2}_{q}}

\newcommand{\loc}{\mathrm{loc}}
\newcommand{\rowsupp}{\mathrm{rsp}}
\newcommand{\colsupp}{\mathrm{csp}}
\newcommand{\sInt}{\mathrm{sInt}}

\makeatletter
\renewcommand{\@makefnmark}{\mbox{\textsuperscript{}}}
\makeatother

\allowdisplaybreaks[1]

\begin{document}

\title{Supercharacter theories of type $A$ unipotent \\ radicals and unipotent polytopes}

\author{Nathaniel Thiem\\ University of Colorado \textbf{Boulder}}

\date{}

\maketitle

\begin{abstract}
Even with the introduction of supercharacter theories, the representation theory of many unipotent groups remains mysterious.   This paper constructs a family of supercharacter theories for normal pattern groups in a way that exhibit many of the combinatorial properties of the set partition combinatorics of the full uni-triangular groups, including combinatorial indexing sets, dimensions, and computable character formulas.  Associated with these supercharacter theories is also a family of polytopes whose integer lattice points give the theories geometric underpinnings.  
\end{abstract}

\section{Introduction}

Supercharacter theory has infused the representation theory of unipotent groups with the combinatorics of set partitions.  Specifically, set partitions index the supercharacters of the maximal unipotent upper-triangular subgroup $\UT$ of the finite general linear group $\GL$ \cite{An,Ya}, and similar theories exist for the maximal unipotent subgroups of other finite reductive groups \cite{AFN,Aw}.  However, while there are  supercharacter theories for other unipotent groups, they do not generally exhibit this computable and combinatorial nature.   This paper seeks to define a natural family of supercharacter theories for the normal pattern subgroups of $\UT$.   As an added bonus, we not only obtain a combinatorial description for these theories, but also gain geometric underpinnings coming from a family of integral polytopes.

Diaconis--Isaacs defined a supercharacter theory of a finite group $G$ as a direct analogue of its character theory, where they replacing conjugacy classes with superclasses and irreducible characters with supercharacters \cite{DI}.  Their approach is based on Andr\'e's adaption of the Kirillov orbit method to study $\UT$, and the underlying axioms are calibrated to preserve as many properties of irreducible characters and conjugacy classes as possible.  For example, the supercharacters are an orthogonal (but not generally orthonormal) basis for the space of functions that are constant on superclasses.  This definition has given us new approaches to groups whose representation theories are known to be difficult (eg. unipotent groups).  Not only can these new theories be combinatorially striking \cite{AIM}, but they can also be used in place of the usual character theory \cite{ADS} in applications, they give a starting point in studying difficult theories \cite{DG}, or give character theoretic foundations for number theoretic identities (eg. \cite{BBFGGKMTS,FoGK}).  

The supercharacter theories of this paper are fundamentally based on Andr\'e's original construction for $\UT$ \cite{An} and Diaconis--Isaacs' later generalization to algebra groups \cite{DI}.  These constructions use two-sided orbits in the dual space $\fkut^*$ of the corresponding Lie algebra $\fkut$ to construct the supercharacters.  In the algebra group case the group $\UT$ acts on $\fkut^*$ by left and right multiplication (technically pre-composition by left and right multiplication on $\fkut$).  In this paper we modify this construction by instead acting by parabolic subgroups of $\GL$.  The resulting theory is coarser but far more combinatorial in nature.    In particular, we obtain statistics such as dimension, nestings and crossings that generalize the corresponding set partition statistics \cite{CPKR}, and in Theorem \ref{FullFormula} we give a character formulas with a ``factorization" analogous to the well-known $\UT$-cases.

For each supercharacter theory there is an associated polytope whose integer lattice points index the supercharacters of the theory.  Thus, the supercharacter theories could in principle give a categorified version of the Ehrhardt polynomials of these polytopes.  These polytopes include all transportation polytopes \cite{DK}, and may be viewed as a family of subfaces of transportation polytopes.  This point of view not only gives a geometric approach to these supercharacter theories, but it also re-interprets set partitions as vertices of a polytope.  Since I am unaware of other contexts where these polytopes may have been studied, I will refer to them as \emph{unipotent polytopes}.  At present we do not understand the significance of this geometry in the representation theory of unipotent groups, and this seems to be a promising direction for future work.

Section \ref{SectionPreliminaries} reviews some of the background material on unipotent groups and supercharacter theories.  Section \ref{SectionGroups} defines the particular unipotent groups we will focus on.  Pattern groups arise naturally in context of groups of Lie type, since they are unipotent groups invariant under conjugation by a maximal split torus.    The unipotent groups we consider are effectively a block-analogue to pattern groups where we define a notion of invariance under the action of a fixed Levi subgroup.  Section \ref{SectionSupercharacterTheories} shows how to use this Levi subgroup to construct supercharacter theories for the group, and shows how the supercharacters (and superclasses) are indexed by the $\ZZ_{\geq 0}$-lattice points of a polytope.  Section \ref{SectionSupercharacterFormulas} computes the supercharacter formulas for these theories, and discusses some consequences of these results.

In addition to teasing out the geometric implications of the underlying polytopes, this paper also gives a framework for studying random walks on the integer lattice points of unipotent polytopes, and representation theoretic Hopf structures on them.  However, both these directions are beyond the scope of this paper.

\subsubsection*{Acknowledgements}

The author would like to thank H. Hausmann for pointing out the connection between the combinatorics of bounded row and column sums and integer polytopes, and to K. Meszaros for directing me to transportation polytopes.  Also, thanks to M. Aguiar for (a) helping work out the corresponding monoid structure and (b) convincing me it is too involved for this paper.
The  work was also partially supported by NSA H98230-13-0203.

\section{Preliminaries} \label{SectionPreliminaries}

This paper is primarily concerned with unipotent subgroups of the finite general linear groups.  It is standard to index the rows and columns of the corresponding matrices by the set $\{1,2,\ldots, N\}$ in the usual total order, but recent work \cite{ABT} has suggested that it is better to instead allow an arbitrary set of size $N$ to index the rows and columns and fix a total order $\cN$ on that set.  However, the paper may be easily read with $\cN$ as the total order $1<2<\cdots < N$. 

This section reviews the relevant unipotent groups, a combinatorial interpretation of normality, and some of the standard supercharacter theories for these groups.  The last section gives a quick refresher of $q$-binomial coefficients.

\subsection{Subgroups of Lie type}

Let $\cN$ be a fixed total order of a finite set with $N$ elements and fix a finite field $\FF_q$ with $q$ elements.  Let $\GL_\cN$ denote the finite general linear group on matrices with rows and columns indexed by our finite set in the order dictated by $\cN$.  If $\mathrm{char}(\FF_q)=p$, then a Sylow $p$-subgroup of $\GL_\cN$ is the subgroup of unipotent upper-triangular matrices 
$$\UT_\cN=\{g\in \GL_\cN\mid (g-\Id_N)_{ij}\neq 0 \text{ implies } i\prec_\cN j\}.$$
The normalizer in $\GL_\cN$ of $\UT_\cN$  is the Borel subgroup 
$$B_\cN=\{g\in \GL_\cN\mid g_{ij}\neq 0 \text{ implies } i\preceq_\cN j\}=N_{\GL_\cN}(\UT_\cN).$$

Let 
$$\fkut_\cN=\UT_\cN-\Id_N$$
denote the corresponding nilpotent $\FF_q$-algebra.  If $\fkn\subseteq \fkut_\cN$ is any subalgebra, then we obtain a subgroup $\Id_N+\fkn\subseteq \UT_\cN$ called and \emph{algebra subgroup}.  If $\cP$ is a subposet of $\cN$ on the same underlying set, then we call the algebra subgroup
$$\UT_\cP=\Id_N+\fkut_\cP\subseteq \UT_\cN, \qquad \text{where} \qquad \fkut_\cP=\{x\in \fkut_\cN\mid x_{ij}\neq 0 \text{ implies } i\prec_{\cP} j\},$$
a \emph{pattern subgroup} of $\UT_\cN$.  Note that transitivity in the poset $\cP$ exactly implies that this $\UT_\cP$ is closed under multiplication.

\subsection{Normal posets}\label{SectionNormalPosets}

In general, a subposet $\cP$ of a poset $\cQ$ does not give a normal subgroup $\UT_\cP$ of $\UT_\cQ$.  However, there is a straight-forward condition on the poset that characterizes this group theoretic condition: a subposet $\cP\subseteq \cQ$ is \emph{normal} if $j\prec_\cP k$ implies $i\prec_\cP k$ and $j\prec_\cP l$ for all $i\prec_{\cQ}j$ and $k\prec_\cQ l$.
In this case, we write $\cP\triangleleft \cQ$.  

Alternatively, if $\sInt(\cQ)$ is the strict interval poset on the set $\{(i,j)\mid i\prec_\cQ j\}$ given by $(j,k)\preceq_{\sInt(\cQ)} (i,l)$ if and only if $i\preceq_\cQ j\prec_\cQ k\preceq_{\cQ} l$, then $\cP$ is normal in $\cQ$ if and only if $\sInt(\cP)$ is a dual order ideal of $\sInt(\cQ)$.  In fact, in this case, $\sInt(\cP)$ exactly gives the coordinates of the Ferrer's shape $F$ or the coordinates allowed to be nonzero in $\UT_{\cP}$.   For example, 
$$
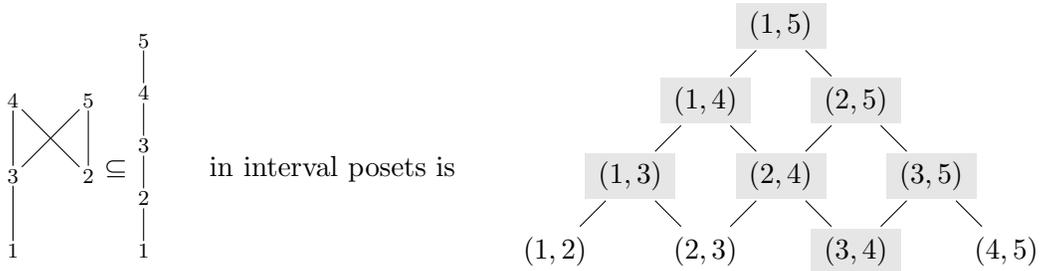
\begin{tikzpicture}[scale=1,baseline=1cm]
\foreach \x/\y/\z in {0/0/1,1/1/2,0/1/3,0/2/4,1/2/5}
	\node (\z) at (\x,\y) [inner sep = 1pt] {$\scs\z$};
\foreach \a/\b in {1/3,2/4,2/5,3/4,3/5}
	\draw (\a) -- (\b);
\end{tikzpicture}
\subseteq
\begin{tikzpicture}[scale=.7,baseline=1cm]
\foreach \x/\y/\z in {0/0/1,0/1/2,0/2/3,0/3/4,0/4/5}
	\node (\z) at (\x,\y) [inner sep = 1pt] {$\scs\z$};
\foreach \a/\b in {1/2,2/3,3/4,4/5}
	\draw (\a) -- (\b);
\end{tikzpicture}
\qquad \text{in interval posets is}\qquad 
\begin{tikzpicture}[scale=1,baseline=2cm]
\foreach \x/\y in {1/3,1/4,1/5,2/4,2/5,3/4,3/5}
	\fill[gray!20!white] (\x+\y,\y-\x)+(-.6,-.3) rectangle ++(.6,.3);
\foreach \x/\y in {1/2,1/3,1/4,1/5,2/3,2/4,2/5,3/4,3/5,4/5}
	\node (\x\y) at (\x+\y,\y-\x) {$(\x,\y)$};
	\draw (12) -- (13) -- (14) -- (15) -- (25) -- (35) -- (45);
	\draw (13) -- (23) -- (24) -- (25);
	\draw (35) -- (34) -- (24) -- (14);
\end{tikzpicture}
$$
where the boxed entries are correspond to the subposet.

There are a number of combinatorial interpretations of normal posets of the total order $\cN$.  For $N\in \ZZ_{\geq 0}$, let $D_N$ denote the Ferrer's shape $(N-1,N-2,\ldots, 1)$, where we right justify the rows.  For example, 
$$D_5=\begin{tikzpicture}[scale=.3,baseline=-.5cm]
\foreach \y in {0,...,3}
	{\foreach \x in {\y,...,3}
		\draw (\x,-\y)+(-.5,-.5) rectangle ++(.5,.5);}
\end{tikzpicture}\ .$$

\begin{proposition}\label{NormalPosetInterpretations}
There are bijections
$$\begin{array}{c@{\ }c@{\ }c@{\ }c@{\ }c}\left\{\begin{array}{@{}c@{}} \text{Dyck paths from}\\ \text{ $(0,0)$ to $(2N,-2N)$} \end{array}\right\} & \longleftrightarrow & \left\{\begin{array}{@{}c@{}} \text{normal sub-}\\ \text{posets  of $\cN$}\end{array}\right\} & \longleftrightarrow & \left\{\begin{array}{@{}c@{}} \text{sub-Ferrers}\\ \text{shapes of $D_N$}\end{array}\right\}\\ 
 d_\cP & \mapsfrom & \cP & \mapsto & F_\cP.\end{array}$$
where $(2i-1,2j-1)$ is NorthEast of $d_\cP$ if and only if $i\prec_{\cP} j$ if and only if $(i,j)\in F_\cP$.
\end{proposition}

\begin{example}
For example, if $2N=10$, then
$$
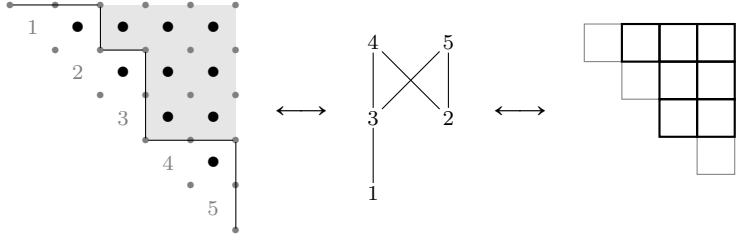
\begin{tikzpicture}[scale=.3,baseline=-1.5cm]
\fill[gray!20!white] (7,-1)+(-3,-1) rectangle ++(3,1);
\fill[gray!20!white] (8,-4)+(-2,-2) rectangle ++(2,2);
\foreach \x in {0,...,5}
	{\foreach \y in {0,...,\x}
		\node[gray] (\x\y) at (2*\x,-2*\y) [inner sep = -1pt] {$\scscs\bullet$};
		}
\foreach \x in {1,...,4}
	{\foreach \y in {1,...,\x}
		\node at (2*\x+1,-2*\y+1) {$\bullet$};	
		}
\foreach \x in {1,...,5}
	\node[gray] at (2*\x-1,-2*\x+1) {$\scs\x$};
\foreach \a/\b/\c/\d in {0/0/2/0,2/0/2/1,2/1/3/1,3/1/3/3,3/3/5/3,5/3/5/5}
	\draw (\a\b)--(\c\d);
\end{tikzpicture}
\quad\longleftrightarrow\quad
\begin{tikzpicture}[scale=1,baseline=1cm]
\foreach \x/\y/\z in {0/0/1,1/1/2,0/1/3,0/2/4,1/2/5}
	\node (\z) at (\x,\y) [inner sep = 1pt] {$\scs\z$};
\foreach \a/\b in {1/3,2/4,2/5,3/4,3/5}
	\draw (\a) -- (\b);
\end{tikzpicture}
\quad\longleftrightarrow \quad
\begin{tikzpicture}[scale=.5,baseline=-1cm]
\foreach \y in {0,...,3}
	{\foreach \x in {\y,...,3}
		\draw[gray] (\x,-\y)+(-.5,-.5) rectangle ++(.5,.5);}
\foreach \x/\y in {1/0,2/0,3/0,2/1,3/1,2/2,3/2}
	\draw[thick] (\x,-\y)+(-.5,-.5) rectangle ++(.5,.5);
\end{tikzpicture}\
$$
where the shaded region accentuates the relevant points NorthEast of the Dyck path.   
\end{example}

\subsection{Supercharacter theories of unipotent groups}\label{SectionSupercharacters}

Supercharacter theories for finite groups were first defined in \cite{DI}, generalizing work by Andr\'e studying representations of $\UT_\cN$ (a series of papers starting with \cite{An}).  There are numerous equivalent formulations of a supercharacter theory, but the following seems most suitable for the purposes of this paper.

A \emph{supercharacter theory} $(\cK,\cX)$ of a finite group $G$ is a pair, where $\cK$ is a partition of $G$ and $\cX$ is a set of characters, such that 
\begin{description}
\item[(SC0)] The number of blocks of $\cK$ is the same as the number of elements in $\cX$.
\item[(SC1)] Each block $K\in \cK$ is a union of conjugacy classes.
\item[(SC2)] The set 
$$\cX\subseteq \{\theta:G\rightarrow \CC\mid \theta(g)=\theta(h), g,h\in K, K\in \cK\}.$$
\item[(SC3)] Each irreducible character of $G$ is the constituent of exactly one element in $\cX$.
\end{description}
We refer to the blocks of $\cK$ as \emph{superclasses} and the elements of $\cX$ as \emph{supercharacters}.

While we have many ways of constructing supercharacter theories, general constructions are not well-understood.  That is, given a finite groups, it is a hard problem to determine its supercharacter theories.   Some groups have remarkably few supercharacter theories, such as the symplectic group $\mathrm{Sp}_6(\FF_2)$ with exactly 2 \cite{BLLW}, and some groups have surprisingly many, such as $C_3\times C_6$ with 297 distinct supercharacter theories.  However, for this paper we follow the basic strategy laid out by \cite{DI} for algebra groups.

Let $\Id_N+\fkn\subseteq \UT_\cN$ be an algebra subgroup.  Then $\Id_N+\fkn$ acts on both $\fkn$ and its vector space dual $\fkn^*$ by left and right multiplication, where
$$(a\cdot y\cdot b)(x)=y(a^{-1}yb^{-1}), \qquad \text{for $a,b\in \Id_N+\fkn$, $x\in \fkn$, $y\in \fkn^*$.}$$
Fix a nontrivial homomorphism $\vartheta:\FF_q^+\rightarrow \GL_1(\CC)\cong \CC^\times$.  In this situation \cite{DI} define a supercharacter theory given by
\begin{description}
\item[$AG$-superclasses of $\Id_N+\fkn$.]   The set partition $\{\Id_N+(\Id_N+\fkn)x(\Id_N+\fkn)\mid x\in \fkn\}$
of $\Id_N+\fkn$.
\item[$AG$-supercharacters of $\Id_N+\fkn$.] The set of characters
$$\Big\{\chi_{AG}^y=\frac{|(\Id_N+\fkn)y|}{|(\Id_N+\fkn)y(\Id_N+\fkn)|} \sum_{z\in (\Id_N+\fkn)y(\Id_N+\fkn)} \vartheta\circ z\mid y\in \fkn^*\Big\}.$$
\end{description}
\begin{remark}
In the case where $\fkn=\fkut_\cN$, this supercharacter theory gives a nice combinatorial theory developed algebraically by Andr\'e \cite{An} and more combinatorially by Yan \cite{Ya}.   However, in general even this supercharacter theory may be wild for algebra subgroups.  In fact, we do not even understand it for pattern subgroups.  
\end{remark}

For the purposes of our generalization, there is a slight coarsening of the $AG$-supercharacter theory for $\UT_\cN$ called the $B_\cN$-supercharacter theory that exists because $\fkut_\cN$ and $\fkut_\cN^*$ are in fact permuted by left and right multiplication by $B_\cN$.
\begin{description}
\item[$B_\cN$-superclasses of $\UT_\cN$.]   The set partition $\{\Id_N+B_\cN x B_\cN\mid x\in \fkut_\cN\}$ of $\UT_\cN$.
\item[$B_\cN$-supercharacters of $\UT_\cN$.] The set of characters
$$\Big\{\chi_{\cN}^y=\sum_{z\in B_\cN y B_\cN} \vartheta\circ z\mid y\in \fkut_\cN^*\Big\}.$$
\end{description}
In this case, the supercharacters and superclasses are indexed by set partitions of the underlying set.  In the following sections it is best to view set partitions as functions $\lambda:\sInt(\cN)\rightarrow \{0,1\}$ such that for $\lambda_{ik}=1=\lambda_{jl}$, we have $i=j$ if and only if $k=l$ (or placements of non-attacking rooks on a triangular chessboard).  The primary purpose of this paper is to generalize this set-up to nice families of pattern groups.  
\begin{remark}
This version of the supercharacters is scaled slightly from the conventional choice.  That is, usually each character $\chi_{\cN}^y$ is multiplied by 
$$ \frac{|\UT_{\cN} y|}{|\UT_{\cN} y \UT_{\cN}|}.$$ 
This scaling removes some excessive multiplicities.  However, in this paper the ``best" scaling factor for the supercharacters below is not clear, so the paper is written with them removed entirely.  This choice implies that in fact
$$\chi_{\cN}^y=\sum_{\psi\in X_y} \psi(1)\psi$$
for some set of irreducible characters $X_y$.
\end{remark}

\subsection{$q$-Binomials and weights}

The symbol $q$ will generally be the size of a finite field, but for this section may treat $q$ as an indeterminate.  For $n,k\in\ZZ_{\geq 0}$, let
$$\qbin{n}{k}=\frac{[n]!}{[k]![n-k]!},\qquad \text{where}\qquad [n]!=[n][n-1]\cdots [2][1] \quad \text{and}\quad [n]=\frac{q^n-1}{q-1}.$$ 
In this subsection, we review some other interpretations and results associated with these $q$-binomial coefficients.  

Let $S_n$ be the symmetric group on $n$ letters.  Then
$$[n]!=\sum_{w\in S_n} q^{\mathrm{inv}(w)},\qquad \text{where}\qquad \mathrm{inv}(w)=\#\{1\leq i<j\leq n\mid w(i)>w(j)\}.$$
In particular, it will be useful to note that since $[n]!$ is palindromic of degree $\binom{n}{2}$,
\begin{equation}\label{InvolutionSum}
\sum_{w\in S_n} \frac{1}{q^{\mathrm{inv}(w)}}=\frac{1}{q^{\binom{n}{2}}} [n]!.
\end{equation}

Let $B\subseteq C$, where $C$ is a set with a total order $\cC$.  Let
$$\begin{array}{rccc} \wt^\uparrow_B:&\{A\subseteq C\} & \longrightarrow & \ZZ_{\geq 0}\\ & A & \mapsto & \#\{(a,b)\in A\times B\mid a\prec_\cC b\},\end{array}$$
and for $A,B\subseteq C$, let
$$\wt^\downarrow_B(A)=\wt^\uparrow_A(B).$$
Then
\begin{equation}\label{WeightqBinomial}
\qbin{n}{k}=q^{\binom{k}{2}}\sum_{A\subseteq \{1,\ldots, n\}\atop |A|=k} q^{\wt_{\{1,\ldots,n\}}^\uparrow (A)}=q^{\binom{k}{2}}\sum_{A\subseteq \{1,\ldots, n\}\atop |A|=k} q^{\wt_{\{1,\ldots,n\}}^\downarrow (A)}.
\end{equation}

\section{A parabolic generalization of pattern subgroups} \label{SectionGroups}

In this section, we build a family of pattern subgroups of $\UT_\cN$; they will all be normal, and each will have a family of supercharacter theories, defined in Section \ref{SectionSupercharacterTheories}.  The choice of a subgroup with an associated supercharacter theory will determine a polytope, giving the theory a geometric foundation.  

\subsection{Parabolic posets and $\UT_\beta$}

We begin by defining a family of unipotent groups that appear naturally in the theory of reductive groups, the unipotent radicals of parabolic subgroups.  It turns out that for $\GL_\cN$, these unipotent groups are pattern groups and their associated posets are easy to characterize.   In the Section \ref{SectionSupercharacterTheories}, each unipotent radical $\UT_\beta$ will determine a family of supercharacter theories.

A subposet $\cQ$ is \emph{parabolic} in $\cN$ if there exists a set composition $(\cQ_1, \cQ_2, \ldots, \cQ_\ell)$ of the underlying set such that $a\prec_\cQ b$ if and only if  $a\in \cQ_i$ and $b\in \cQ_j$ with $i<j$.  These subposets are necessarily normal, where the corresponding Dyck path always returns down to the diagonal before moving right again.  We will write $\cQ\triangleleft_\parb \cN$.  For example, if $\cN$ is $1<2<3<4$, then the parabolic subposets (with associated set compositions) are 
$$\begin{array}{ccccccc}
\begin{tikzpicture}[scale=.5,baseline=.75cm]
\foreach \x/\y/\z in {0/0/1,0/1/2,0/2/3,0/3/4}
	\node (\z) at (\x,\y) [inner sep = 1pt] {$\z$};
\foreach \y/\z in {1/2,2/3,3/4}
	\draw (\y) -- (\z);
\end{tikzpicture}\ , &
\begin{tikzpicture}[scale=.5,baseline=.5cm]
\foreach \x/\y/\z in {0/0/1,2/0/2,1/1/3,1/2/4}
	\node (\z) at (\x,\y) [inner sep = 1pt] {$\z$};
\foreach \y/\z in {1/3,2/3,3/4}
	\draw (\y) -- (\z);
\end{tikzpicture}, &
\begin{tikzpicture}[scale=.5,baseline=.5cm]
\foreach \x/\y/\z in {1/0/1,0/1/2,2/1/3,1/2/4}
	\node (\z) at (\x,\y) [inner sep = 1pt] {$\z$};
\foreach \y/\z in {1/2,1/3,2/4,3/4}
	\draw (\y) -- (\z);
\end{tikzpicture}, &
\begin{tikzpicture}[scale=.5,baseline=.5cm]
\foreach \x/\y/\z in {1/0/1,1/1/2,0/2/3,2/2/4}
	\node (\z) at (\x,\y) [inner sep = 1pt] {$\z$};
\foreach \y/\z in {1/2,2/3,2/4}
	\draw (\y) -- (\z);
\end{tikzpicture}, &
\begin{tikzpicture}[scale=.5,baseline=.25cm]
\foreach \x/\y/\z in {0/0/1,1/0/2,2/0/3,1/1/4}
	\node (\z) at (\x,\y) [inner sep = 1pt] {$\z$};
\foreach \y/\z in {1/4,2/4,3/4}
	\draw (\y) -- (\z);
\end{tikzpicture},
&
\begin{tikzpicture}[scale=.5,baseline=.25cm]
\foreach \x/\y/\z in {1/0/1,0/1/2,1/1/3,2/1/4}
	\node (\z) at (\x,\y) [inner sep = 1pt] {$\z$};
\foreach \y/\z in {1/2,1/3,1/4}
	\draw (\y) -- (\z);
\end{tikzpicture},
&
\begin{tikzpicture}[scale=.5,baseline=0cm]
\foreach \x/\y/\z in {0/0/1,1/0/2,2/0/3,3/0/4}
	\node (\z) at (\x,\y) [inner sep = 1pt] {$\z$};
\end{tikzpicture}.\\
\scs(\{1\},\{2\},\{3\},\{4\},\{5\}) & \scs (\{1,2\},\{3\},\{4\}) & \scs(\{1\},\{2,3\},\{4\}) & \scs (\{1\},\{2\},\{3,4\}) & \scs(\{1,2,3\},\{4\}) & \scs(\{1\},\{2,3,4\}) & \scs (\{1,2,3,4\})

\end{array}
$$
Since given a total order $\cN$ the sizes of the blocks of the set composition completely determines $\cQ$, we deduce the following proposition.

\begin{proposition}\label{BoundaryMap}
There is a bijection
$$\begin{array}{r@{\ }c@{\ }c@{\ }c}\bdry:  &\left\{\begin{array}{@{}c@{}} \text{Parabolic sub-}\\ \text{posets of $\cN$}\end{array}\right\} & \longrightarrow & \left\{\begin{array}{@{}c@{}}\text{integer com-}\\ \text{positions of $N$}\end{array}\right\}\\ 
&\cQ & \mapsto & (|\cQ_1|,\ldots, |\cQ_\ell|).\end{array}$$
\end{proposition}

For an integer composition $\beta\vDash N$ and an underlying total order $\cN$, define
$$\UT_{\beta}=\UT_{\bdry^{-1}(\beta)}$$
(note that $\bdry^{-1}(\beta)$ makes no sense without $\cN$).

Every parabolic subposet $\cQ$ in $\cN$ with $\beta=\bdry(\cQ)$ has a corresponding Levi subgroup
$$L_{\beta}=\left[\begin{array}{c|c|c|c} 
\GL_{\beta_1} & 0 & \cdots & 0\\ \hline 
0 &  \GL_{\beta_2}  & \ddots & \vdots\\ \hline  
\vdots & \ddots & \ddots& 0\\ \hline 
0 & \cdots & 0 &  \GL_{\beta_\ell}\end{array}\right],$$
such that  $\UT_\beta$ is the unipotent radical of the parabolic subgroup
$$P_\beta=L_{\beta}\ltimes \UT_\beta=N_{\GL_\cN(q)}(\UT_\beta).$$

\begin{remark}
The Lie theoretic language of parabolics, Levis and unipotent radicals is merely given for context.  The reader is welcome to ignore the terminology and focus on the definitions.
\end{remark}

\subsection{Levi compatible posets and $\UT_{(\beta,\cP)}$}

This section defines a family of subgroups of $\UT_\beta$ in a way that gives a ``block" analogue to pattern groups.  In the ensuing sections we will primarily be interested in the subgroups of this nature that are in fact normal in $\UT_\beta$.

A \emph{Levi compatible} subposet $\cP$ of $\cQ\triangleleft_\parb \cN$ is a  subposet such that 
$$L_{\bdry(\cQ)}\subseteq N_{\GL_\cN(q)}(\UT_\cP).$$
In this case, we write $\cP\subseteq_\levi \cQ$.

\begin{proposition} \label{FatMap} If $\cQ$ is parabolic with set composition $(\cQ_1,\cQ_2,\ldots,\cQ_\ell)$ and $\beta=\bdry(\cQ)$, then there is a bijection 
$$\begin{array}{r@{\ }c@{\ }c@{\ }c}\fat_\beta:  &\left\{\begin{array}{@{}c@{}} \text{Subposets of}\\ \text{$1<2<\cdots<\ell$}\end{array}\right\} & \longrightarrow & \left\{\begin{array}{@{}c@{}}\text{Levi compatible}\\ \text{subposets of $\cQ$}\end{array}\right\}\\ 
& \cP & \mapsto & \fat_\beta(\cP).\end{array}$$
where 
$$a\prec_{\fat_\beta(\cP)}b\qquad \text{if and only if}\qquad \text{$a\in \cQ_i$, $b\in \cQ_j$ with $i\prec_{\cP} j$.}$$
\end{proposition}

For $\fat_\beta(\cP)\subseteq_\levi \bdry^{-1}(\beta)$, let
$$\UT_{(\beta,\cP)}= \UT_{\fat_\beta(\cP)},$$
so that $\UT_\beta=\UT_{(\beta,\cL)}$ where $\cL$ is the usual total order on $\{1,2,\ldots, \ell\}$.

\begin{examples}\hfill

\begin{enumerate}
\item[(E1)]
All subposets are Levi compatible with the total order $\cN$.  This notion therefore gives a generalization of ``pattern subgroups" to arbitrary unipotent radicals of parabolics.
\item[(E2)]  If $\cN$ is $1<2<3<4$, then 
$$
\begin{tikzpicture}[scale=.5,baseline=0cm]
\fill[gray!20!white]  (1.5,0) +(-.9,-0.4) rectangle ++(.9,0.4);
\foreach \x/\y/\z in {0/0/1,1/0/2,2/0/3,3/0/4}
	\node (\z) at (\x,\y) [inner sep = 1pt] {$\z$};
\end{tikzpicture},
\quad
\begin{tikzpicture}[scale=.5,baseline=.2cm]
\fill[gray!20!white] (1.5,0) +(-.9,-0.4) rectangle ++(.9,0.4);
\foreach \x/\y/\z in {0/0/1,1/0/2,2/0/3,1/1/4}
	\node (\z) at (\x,\y) [inner sep = 1pt] {$\z$};
\foreach \y/\z in {1/4,2/4,3/4}
	\draw (\y) -- (\z);
\end{tikzpicture},
\quad
\begin{tikzpicture}[scale=.5,baseline=.2cm]
\fill[gray!20!white](.5,1) +(-.9,-0.4) rectangle ++(.9,0.4);
\foreach \x/\y/\z in {1/0/1,0/1/2,1/1/3,2/1/4}
	\node (\z) at (\x,\y) [inner sep = 1pt] {$\z$};
\foreach \y/\z in {1/2,1/3,1/4}
	\draw (\y) -- (\z);
\end{tikzpicture},
\quad
\begin{tikzpicture}[scale=.5,baseline=.2cm]
\fill[gray!20!white] (1,1) +(-1.4,-0.4) rectangle ++(1.4,0.4);
\foreach \x/\y/\z in {1/0/1,0/1/2,2/1/3,3/0/4}
	\node (\z) at (\x,\y) [inner sep = 1pt] {$\z$};
\foreach \y/\z in {1/2,1/3}
	\draw (\y) -- (\z);
\end{tikzpicture},
\quad
\begin{tikzpicture}[scale=.5,baseline=.2cm]
\fill[gray!20!white] (1.5,0) +(-.9,-0.4) rectangle ++(.9,0.4);
\foreach \x/\y/\z in {0/0/1,1/0/2,2/0/3,0/1/4}
	\node (\z) at (\x,\y) [inner sep = 1pt] {$\z$};
\foreach \y/\z in {1/4}
	\draw (\y) -- (\z);
\end{tikzpicture},
\quad
\begin{tikzpicture}[scale=.5,baseline=.2cm]
\fill[gray!20!white] (2,0) +(-1.4,-0.4) rectangle ++(1.4,0.4);
\foreach \x/\y/\z in {0/0/1,1/0/2,3/0/3,2/1/4}
	\node (\z) at (\x,\y) [inner sep = 1pt] {$\z$};
\foreach \y/\z in {2/4,3/4}
	\draw (\y) -- (\z);
\end{tikzpicture}\subseteq_\levi 
\begin{tikzpicture}[scale=.5,baseline=.2cm]
\foreach \x/\y/\z in {1/0/1,0/1/2,2/1/3,1/2/4}
	\node (\z) at (\x,\y) [inner sep = 1pt] {$\z$};
\foreach \y/\z in {1/2,1/3,2/4,3/4}
	\draw (\y) -- (\z);
\end{tikzpicture}.
$$
\end{enumerate}
where the shaded elements are treated as a single element.
\end{examples}

\begin{remark}The relation $\subseteq_\levi$ gives a partial order on the on the set of parabolic subposets of a fixed $\cN$.  In this poset, $\cN$ is the unique maximal element, and the totally unrelated set is the unique minimal element.
\end{remark}

\subsection{Unipotent polytopes}

Fix an integer composition $\beta=(\beta_1,\ldots,\beta_\ell)\vDash N$ and let $\cP$ be a normal subposet of $1<2<\cdots<\ell$ with corresponding Ferrer's shape $F$.  The \emph{unipotent polytope} $(\beta,\cP)$ is the convex polytope in the positive quadrant $\RR_{\geq 0}^{|F|}$ determined by the inequalities
$$\Big\{\sum_{i\prec_\cP j} x_{ij}\leq \beta_j,\sum_{j\prec_\cP k} x_{jk}\leq \beta_j\ \Big|\ 1\leq j\leq \ell\Big\}.$$

\begin{remark} If $F$ is the Ferrer's shape corresponding to $\cP$, then one may view the unipotent polytope as possible fillings of the boxes of $F$ by non-negative real numbers such that the row and column sums are bounded by $\beta$.
\end{remark}

\begin{examples}\hfill

\begin{enumerate}
\item[(E1)] If $\beta=(2,3,1,1,5)$, and
 $$\cP=
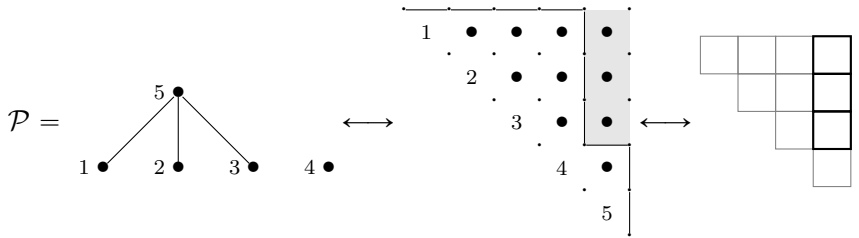
\begin{tikzpicture}[scale=1,baseline=.5cm]
	\foreach \x/\y/\z in {0/0/1,1/0/2,2/0/3,3/0/4,1/1/5}
		{\node (\x\y) at (\x,\y) [inner sep=-1pt] {$\bullet$};
		\node at (\x-.25,\y) {$\scs\z$};}
	\foreach \x in {0,1,2}
		\draw (\x0) -- (11); 
\end{tikzpicture}
\longleftrightarrow 
\begin{tikzpicture}[scale=.3,baseline=-1.6cm]
	\fill[gray!20!white] (9,-3)+(-1,-3) rectangle ++(1,3);
	\foreach \y in {0,...,5}
		{\foreach \x in {\y,...,5}
			\node (\x\y) at (2*\x,-2*\y) [inner sep=-4pt] {$\cdot$};
			}
	\foreach \y in {0,...,3}
		{\foreach \x in {\y,...,3}
			\node at (2*\x+3,-2*\y-1) [inner sep=-4pt] {$\bullet$};
			}
	\foreach \x/\y in {00/10,10/20,20/30,30/40,40/41,41/42,42/43,43/53,53/54,54/55}
		\draw (\x) -- (\y);
	\foreach \x in {1,...,5}
		\node at (2*\x-1,-2*\x+1) {$\scs\x$};
\end{tikzpicture}\longleftrightarrow 
\begin{tikzpicture}[scale=.5,baseline=-1cm]
\foreach \y in {0,...,3}
	{\foreach \x in {\y,...,3}
		\draw[gray] (\x,-\y)+(-.5,-.5) rectangle ++(.5,.5);}
\foreach \x/\y in {3/0,3/1,3/2}
	\draw[thick] (\x,-\y)+(-.5,-.5) rectangle ++(.5,.5);
\end{tikzpicture}\ ,$$
 then the equations $x_{15}\leq 2,
x_{25}\leq  3,
x_{35}\leq  1,
x_{15}+x_{25}+x_{35}\leq 5$ give the polytope
$$
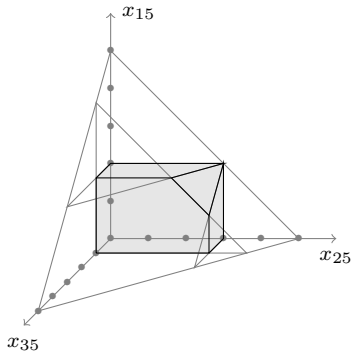
\begin{tikzpicture}[scale=.5,baseline=.5cm]
	\fill[gray!20!white] (0,0,1) -- (3,0,1) -- (3,1,1) -- (2,2,1) -- (0,2,1) -- (0,0,1);
	\fill[gray!20!white] (3,0,1) -- (3,0,0) -- (3,2,0) -- (3,1,1) -- (3,0,1);
	\fill[gray!20!white]  (0,2,1) -- (0,2,0) -- (3,2,0) -- (2,2,1) -- (0,2,1);
	\fill[gray!20!white]  (2,2,1) -- (3,2,0) -- (3,1,1) -- (2,2,1);
	\foreach \x in {1,...,5}
		{\node[gray] (\x00) at (\x,0,0) [inner sep=-1pt] {$\scscs\bullet$};
		\node[gray] (0\x0) at (0,\x,0) [inner sep=-1pt] {$\scscs\bullet$};
		\node[gray] (00\x) at (0,0,\x) [inner sep=-1pt] {$\scscs\bullet$};}
	\node[gray] (000) at (0,0,0) [inner sep=-1pt] {$\scscs\bullet$};
	\draw[gray] (5,0,0) -- (0,5,0) -- (0,0,5) -- (5,0,0);
	\draw[->,gray] (0,0,0) -- (6,0,0);
	\draw[->,gray] (0,0,0) -- (0,6,0);
	\draw[->,gray] (0,0,0) -- (0,0,6);
	\draw[gray] (0,0,1) -- (4,0,1) -- (0,4,1) -- (0,0,1);
	\draw[gray] (3,0,0) -- (3,0,2) -- (3,2,0) -- (3,0,0);
	\draw[gray] (0,2,0) -- (3,2,0) -- (0,2,3) -- (0,2,0);
	\draw (0,0,1) -- (3,0,1) -- (3,1,1) -- (2,2,1) -- (0,2,1) -- (0,0,1);
	\draw (3,0,1) -- (3,0,0) -- (3,2,0) -- (3,1,1);
	\draw (0,2,1) -- (0,2,0) -- (3,2,0) -- (2,2,1);
	\node at (.75,6,0) {$\scs x_{15}$};
	\node at (6,-.5,0) {$\scs x_{25}$};
	\node at (0,-.5,6) {$\scs x_{35}$};
\end{tikzpicture}\ .
$$
\item[(E2)] In general, the unipotent polytopes $(\beta,\cP)$ where the corresponding Ferrer's shape $F$ is a rectangle are called transportation polytopes.  They are a subfamily corresponding to abelian unipotent groups.  In this case, the bounds on the row sums and and the bounds on the column sums are independent.
\end{enumerate}
\end{examples}

Unipotent polytopes match up with Levi compatible sub-posets in the following way.

\begin{proposition} Let $\cN$ be a total order on a set with $N$ elements and let $\beta\vDash N$.  Then the following are equivalent:
\begin{enumerate}
\item[(P1)] $(\beta,\cP)$ is a unipotent polytope,
\item[(P2)] $\fat_{\beta}(\cP)\triangleleft_\levi \bdry^{-1}(\beta)$,
\item[(P3)] $P_\beta\subseteq N_{\GL_\cN(q)}(\UT_{(\beta,\cP)}).$
\end{enumerate}
\end{proposition}
\begin{proof} Let $\cQ=\bdry^{-1}(\beta)$, and let  $\cL$ be the usual total order on $\{1,\ldots, \ell(\beta)\}$.

 If $(\beta,\cP)$ is a unipotent polytope, then by definition $\cP\triangleleft \cL$.  But then by Proposition \ref{FatMap}, $\fat_\beta(\cP)\subseteq_\levi\fat_\beta(\cL)=\cQ$, and $\fat_\beta$ preserves normality.    

If $\fat_\beta(\cP)\triangleleft_\levi \cQ$, then $\UT_{(\beta,\cP)}\triangleleft \UT_\beta$ and $L_\beta\subseteq N_{\GL_\cN(q)}(\UT_{(\beta,\cP)})$.  Thus, 
$$P_\beta=L_\beta\ltimes \UT_\beta\subseteq N_{\GL_\cN(q)}(\UT_{(\beta,\cP)}).$$

If  $P_\beta\subseteq N_{\GL_\cN(q)}(\UT_{(\beta,\cP)}),$ then $\UT_{(\beta,\cP)}\triangleleft\UT_\beta$, so $\fat_\beta(\cP)\triangleleft  \cQ=\fat_\beta(\cL)$ and since $\fat_\beta$ preserves normality, $\cP\triangleleft \cL$.
\end{proof}

\begin{remark}
For a fixed total order $\cN$ with $N$ elements, the function
$$\begin{array}{c@{\ }c@{\ }c} \left\{\begin{array}{@{}c@{}} \text{unipotent polytopes}\\ \text{$(\beta,\cP)$ with $|\beta|=N$}\end{array} \right\} & \longrightarrow &  \left\{\begin{array}{@{}c@{}} \text{normal sub-}\\ \text{groups of $\UT_\cN$}\end{array} \right\} \\
(\beta,\cP) & \mapsto & \UT_{(\beta,\cP)}\end{array}$$
is not injective, since, for example,
$$\UT_{((1^4),
\begin{tikzpicture}[scale=.5,baseline=.2cm]
	\foreach \x/\y/\z in {0/0/1,1/0/2,0/1/3,1/1/4} 
		\node (\x\y) at (\x,\y) [inner sep=1pt] {$\scscs\z$};
	\draw (00) -- (01);
	\draw (00) -- (11);
	\draw (10) -- (01);
	\draw (10) -- (11);
\end{tikzpicture})}=\left[\begin{array}{cccc} 1 & 0 & \ast & \ast\\
0 & 1 & \ast & \ast\\
0 & 0 & 1 & 0\\
0 & 0 & 0 & 1\end{array}\right]=\UT_{((2,2), \begin{tikzpicture}[scale=.5,baseline=.2cm]
	\foreach \x/\y/\z in {0/0/1,0/1/2} 
		\node (\x\y) at (\x,\y) [inner sep=1pt] {$\scscs\z$};
	\draw (00) -- (01);
\end{tikzpicture})}.
$$
On the other hand, for a fixed $\beta\vDash N$
$$\begin{array}{c@{\ }c@{\ }c} \left\{\begin{array}{@{}c@{}} \text{unipotent poly-}\\ \text{topes $(\beta,\cP)$}\end{array} \right\} & \longrightarrow &  \left\{\begin{array}{@{}c@{}} \text{normal sub-}\\ \text{groups of $\UT_\cN$}\end{array}\right\} \\
(\beta,\cP) & \mapsto & \UT_{(\beta,\cP)}\end{array}$$
is injective.
\end{remark}

\section{Families of parabolic supercharacter theories}\label{SectionSupercharacterTheories}

The data in a unipotent polytope $(\beta,\cP)$ also gives a natural supercharacter theory to a corresponding unipotent group $\UT_{(\beta,\cP)}$.  This section describes this theory, and shows that the supercharacters/superclasses are indexed by the integer lattice points contained in the polytope $(\beta,\cP)$.

\subsection{Supercharacter theory definition}

Let $(\beta,\cP)$ be a unipotent polytope with $\beta\vDash N$. Then
$$\fkut_{(\beta,\cP)}=\UT_{(\beta,\cP)}-\Id_{N}$$
is an $\FF_q$-vector space with basis
$$\{e_{ij}\mid i\prec_{\fat_\beta(\cP)} j\},\qquad \text{where}\qquad (e_{ij})_{kl}=\delta_{(i,j),(k,l)}.$$
The dual space
$$\fkut_{(\beta,\cP)}^*=\Hom_{\FF_q}(\fkut_{(\beta,\cP)},\FF_q)$$
has dual basis
$$\left\{\begin{array}{r@{\ }c@{\ }c@{\ }c}e_{ij}^*:&\fkut_{(\beta,\cP)} &\rightarrow &\FF_q\\ & x & \mapsto & x_{ij}\end{array}\ \bigg| i\prec_{\fat_\beta(\cP)} j\right\}.$$

The group $P_\beta$ acts on both $\fkut_{(\beta,\cP)}$ and $\fkut_{(\beta,\cP)}^*$ by left and right multiplication, where
$$(a\cdot y \cdot b)(x)=y(a^{-1}xb^{-1}), \quad \text{for $x\in \fkut_{(\beta,\cP)}, y\in \fkut_{(\beta,\cP)}^*, a,b\in P_\beta$}.$$
These actions give a natural supercharacter theory for $\UT_{(\beta,\cP)}$.  
\begin{description}
\item[$P_\beta$-superclasses of $\UT_{(\beta,\cP)}$.] The set partition $\{P_\beta x  P_\beta+\Id_{N}\mid x\in \fkut_{(\beta,\cP)}\}$
of $\UT_{(\beta,\cP)}$.
\item[$P_\beta$-supercharacters of $\UT_{(\beta,\cP)}$.]  The characters
\begin{equation} \label{SupercharacterDefinition}
\{\chi_{\beta}^{y}\mid P_\beta y P_\beta\in P_\beta\backslash \fkut_{(\beta,\cP)}^*/P_\beta\}, \quad \text{where}\quad \chi_{\beta}^{y}=\sum_{z\in P_\beta y P_\beta}\vartheta\circ z.\end{equation}
\end{description}

\begin{proposition}
If $(\beta,\cP)$ is a unipotent polytope, then the $P_\beta$-superclasses and the $P_\beta$-super-characters form a supercharacter theory of $\UT_{(\beta,\cP)}$.
\end{proposition}

\begin{proof}
The following conditions are straight-forward to check:
\begin{description}
\item[(SC0)] The fact that the number of $G$ orbits on a vectors space $V$ is the same as the number of $G$-orbits on $V^*$  shows that there are an equal number of superclasses and supercharacters (see \cite[Lemma 4.1]{DI} for an analogous argument).
\item[(SC1)] The superclasses are unions of conjugacy classes.
\item[(SC2)] The supercharacters are constant on superclasses.
\end{description}
It therefore suffices to show that $P_\beta$-supercharacters are in fact orthogonal characters (rather than just class functions) that collectively have the irreducibles of $\UT_{(\beta,\cP)}$ as constituents.  By definition,
\begin{align}
\chi_{\beta}^y & = \sum_{z\in P_\beta y P_\beta} \vartheta\circ z \notag\\
&=\sum_{\UT_{(\beta,\cP)}z' \UT_{(\beta,\cP)}\in P_\beta (\UT_{(\beta,\cP)}y \UT_{(\beta,\cP)}) P_\beta}\bigg( \sum_{z\in \UT_{(\beta,\cP)}z' \UT_{(\beta,\cP)}}\vartheta\circ z\bigg)\notag\\
&= \sum_{\UT_{(\beta,\cP)}z' \UT_{(\beta,\cP)}\in P_\beta (\UT_{(\beta,\cP)}y \UT_{(\beta,\cP)}) P_\beta} \frac{|\UT_{(\beta,\cP)} z' \UT_{(\beta,\cP)}|}{|\UT_{(\beta,\cP)} z'|}\chi_{AG}^{z'}.\label{AGSupercharacterDecomposition}
\end{align}
Thus, $\chi_{\beta}^y$ is a $\ZZ_{\geq 0}$-linear combination of characters and so is also a character.  Furthermore, since each $AG$-supercharacter appears in a unique $P_\beta$-supercharacter, the  $P_\beta$-supercharacters are orthogonal and collectively have all the irreducibles as constituents.
\end{proof}

\begin{remarks}\hfill

\begin{enumerate}
\item[(R1)]  The unipotent polytope $((1^N),\cN)$ recovers the $B_\cN$-supercharacter theory of $\UT_\cN$ from Section \ref{SectionSupercharacters}  (also studied in \cite{BT}). 
\item[(R2)]  Since $B_\cN=P_{(1^{N})}\subseteq P_\beta$, we can coarsen (\ref{AGSupercharacterDecomposition}) to get
\begin{equation} \label{BSupercharacterDecomposition}
\chi_{\beta}^y = \sum_{P_{(1^{N})}z' P_{(1^{N})}\subseteq  P_\beta y  P_\beta}  \chi_{(1^{N})}^{z'}.
\end{equation}
This formula will be more useful below since the $\chi_{(1^{N})}^{z'}$ have nicer character formulas.
\item[(R3)]  In practice, one often scales the supercharacters by some factor that divides the multiplicities of the irreducible constituents.  In this case, there does not seem to be an obvious choice, so we have omitted the scaling factor.  However,  (\ref{AGSupercharacterDecomposition}) implies that one may divide by
$$\frac{|\UT_{(\beta,\cP)} y \UT_{(\beta,\cP)}|}{|\UT_{(\beta,\cP)} y|}$$
and still have characters.
\item[(R4)] An advantage of our definition of the supercharacters (without any scaling) is that it is easy to construct the corresponding modules.  For $y\in \fkut_{(\beta,\cP)}^*$, define the $\UT_{(\beta,\cP)}$-module $M^y$ by a $\CC$-basis
$$\{\boxed{z}\mid z\in P_\beta y P_\beta\}$$
with an action
$$u \cdot\boxed{z}= \vartheta\circ z(u^{-1}-\Id_{|\beta|}) \boxed{uz} \qquad \text{for $u\in \UT_{(\beta,\cP)}$, $z\in P_\beta y P_\beta$}.$$
\item[(R5)] The paper \cite{MT} observes that when a pattern subgroup $\UT_\cP$ is normal in $\UT_\cN$, then it is in fact a union of $AG$-superclasses (in some sense ``supernormal").  In this sense, the $P_\beta$-supercharacter theory makes $\UT_{(\beta,\cP)}$ a supernormal subgroup of $\UT_\beta$.
\end{enumerate}
\end{remarks}

\subsection{The combinatorics of the indexing sets}

For a unipotent polytope $(\beta,\cP)$ with $F_\cP$ as in Proposition \ref{NormalPosetInterpretations}, let
$$\cT_\cP^\beta= \Big\{\begin{array}{@{}r@{}c@{\ }c @{\ }c@{}}\lambda: & F_\cP &\rightarrow &\ZZ_{\geq 0}\\ & (i,j) & \mapsto & \lambda_{ij}\end{array}\ \Big|\ \sum_{k\atop (j,k)\in F_\cP} \lambda_{jk}, \sum_{i\atop (i,j)\in F_\cP} \lambda_{ij}\leq \beta_j,1\leq j\leq \ell\Big\}$$
be the set of $\ZZ_{\geq 0}$-lattice points contained in or on $(\beta,\cP)$.

\begin{examples} \hfill

\begin{enumerate}
\item[(E1)] The set
$$\cT_{
\begin{tikzpicture}[scale=.3]
\foreach \x/\y/\z in {0/0/1,2/0/2,1/1/3}
	\node (\z) at (\x,\y) [inner sep=1pt] {$\scscs \z$};
\foreach \a/\b in {1/3,2/3}
	\draw (\a) -- (\b);
\end{tikzpicture}
}^{(4,1,2)}=\left\{
\begin{tikzpicture}[scale=.4,baseline=0cm]
	\foreach \x/\y/\z in {-1/1/4,0/0/1,1/-1/2}
		{\fill[gray!20!white] (\x,\y) +(-.5,-.5) rectangle ++(.5,.5);
		\node at (\x,\y) {$\scs\z$};}
	\foreach \x/\y/\num in 
	{	1/1/0,
		1/0/0}
{
  \draw (\x,\y) +(-0.5,-0.5) rectangle ++(0.5,0.5);
  \draw (\x,\y) node {$\scs\num$};
}
\end{tikzpicture},\
\begin{tikzpicture}[scale=.4,baseline=0cm]
	\foreach \x/\y/\z in {-1/1/4,0/0/1,1/-1/2}
		{\fill[gray!20!white] (\x,\y) +(-.5,-.5) rectangle ++(.5,.5);
		\node at (\x,\y) {$\scs\z$};}
	\foreach \x/\y/\num in 
	{	1/1/1,
		1/0/0}
{
  \draw (\x,\y) +(-0.5,-0.5) rectangle ++(0.5,0.5);
  \draw (\x,\y) node {$\scs\num$};
}
\end{tikzpicture},\
\begin{tikzpicture}[scale=.4,baseline=0cm]
	\foreach \x/\y/\z in {-1/1/4,0/0/1,1/-1/2}
		{\fill[gray!20!white] (\x,\y) +(-.5,-.5) rectangle ++(.5,.5);
		\node at (\x,\y) {$\scs\z$};}
	\foreach \x/\y/\num in 
	{	1/1/0,
		1/0/1}
{
  \draw (\x,\y) +(-0.5,-0.5) rectangle ++(0.5,0.5);
  \draw (\x,\y) node {$\scs\num$};
}
\end{tikzpicture},\
\begin{tikzpicture}[scale=.4,baseline=0cm]
	\foreach \x/\y/\z in {-1/1/4,0/0/1,1/-1/2}
		{\fill[gray!20!white] (\x,\y) +(-.5,-.5) rectangle ++(.5,.5);
		\node at (\x,\y) {$\scs\z$};}
	\foreach \x/\y/\num in 
	{	1/1/1,
		1/0/1}
{
  \draw (\x,\y) +(-0.5,-0.5) rectangle ++(0.5,0.5);
  \draw (\x,\y) node {$\scs\num$};
}
\end{tikzpicture},\
\begin{tikzpicture}[scale=.4,baseline=0cm]
	\foreach \x/\y/\z in {-1/1/4,0/0/1,1/-1/2}
		{\fill[gray!20!white] (\x,\y) +(-.5,-.5) rectangle ++(.5,.5);
		\node at (\x,\y) {$\scs\z$};}
	\foreach \x/\y/\num in 
	{	1/1/2,
		1/0/0}
{
  \draw (\x,\y) +(-0.5,-0.5) rectangle ++(0.5,0.5);
  \draw (\x,\y) node {$\scs\num$};
}
\end{tikzpicture},\
\right\}$$
where the entries shaded in gray give the bounds for each row and column.
\item[(E2)] If $\beta=(1^{m},n)$ with $m\leq n$ and 
$$\cP=
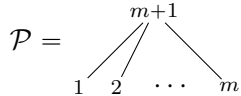
\begin{tikzpicture}[scale=.5,baseline=.5cm]
\foreach \x/\y/\z in {0/0/1,1/0/2,4/0/{m},2/2/{m+1}}
	\node (\x) at (\x,\y) [inner sep=1pt] {$\scs \z$};
\foreach \a/\b in {0/2,1/2,4/2}
	\draw (\a) -- (\b);
\node at (2.5,0) {$\cdots$};
\end{tikzpicture}
$$
then $\cT_\cP^{\beta}$ is the set of vertices of the $m$-dimensional hypercube.
\item[(E3)] If $N=2m$, and
$$\cP=
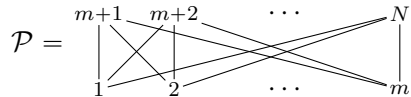
\begin{tikzpicture}[scale=.5,baseline=.5cm]
\foreach \x/\y/\z in {0/0/1,1/0/2,4/0/{m},0/2/{m+1},1/2/{m+2}, 4/2/N}
	\node (\x\y) at (2*\x,\y) [inner sep=1pt] {$\scs \z$};
\foreach \a/\b in {00/02,00/12,00/42,10/02,10/12,10/42,40/02,40/12,40/42}
	\draw (\a) -- (\b);
\node at (5,0) {$\cdots$};
\node at (5,2) {$\cdots$};
\end{tikzpicture}$$
then $\cT_\cP^{(1^N)}$ is the usual basis for the rook monoid.
\item[(E4)] If $\cN$ is a total order of a set $A$, then the set $\cT_\cN^{(1^N)}$ is in bijection with the set of set partitions of $A$.
\end{enumerate}
\end{examples}

\begin{remark} Note that $\fat_{(1^N)}(\cP)=\cP$ and we may in fact identify $\cT_\cP^{(1^N)}$ with subsets of $\fkut_{((1^N),\cP)}$ and $\fkut_{((1^N),\cP)}^*$.  In particular, fix injections
\begin{equation}\label{IndexSpaceInjection}
\begin{array}{ccc} \cT_\cP^{(1^N)} & \longrightarrow & \fkut_{((1^N),\cP)}\\
\mu & \mapsto & \dd e_\mu=\sum_{i\prec_\cP j} \mu_{ij}e_{ij}\end{array} \qquad \text{and}\qquad 
\begin{array}{ccc} \cT_\cP^{(1^N)} & \longrightarrow & \fkut_{((1^N),\cP)}^*\\
\lambda & \mapsto & \dd e_\lambda^*=\sum_{i\prec_\cP j} \lambda_{ij}e_{ij}^*.\end{array}
\end{equation}
\end{remark}

The following theorem establishes the connection between $P_\beta$-supercharacter theories and $\ZZ_{\geq 0}$-lattice points in unipotent polytopes.

\begin{theorem}\label{SuperclassIndexing}
For $(\beta,\cP)$ a unipotent polytope,
$$\left\{\begin{array}{c} \text{$P_\beta$-superclasses}\\ \text{of $\UT_{(\beta,\cP)}$}\end{array}\right\}\longleftrightarrow \cT_\cP^\beta\longleftrightarrow\left\{\begin{array}{c} \text{$P_\beta$-supercharacters}\\ \text{of $\UT_{(\beta,\cP)}$}\end{array}\right\}.$$
\end{theorem}
\begin{proof}  The number of $P_\beta$-supercharacters and $P_\beta$-superclasses is the same, so this proof focuses on the left bijection.

 Let $\cN$ be the underlying total order and let $\cQ=\bdry^{-1}(\beta)$ with corresponding set composition $(\cQ_1,\ldots, \cQ_\ell)$.  
Since $B_\cN=P_{(1^{|\beta|})}\subseteq P_\beta$, every $P_\beta$-superclass is a union of $B_\cN$-superclasses (which are indexed by set partitions, as described in the last remark of Section \ref{SectionSupercharacters}).  That is, for each $u\in \UT_{(\beta,\cP)}$, there exists a subset $\cA_u\subseteq \cT_\cP^{(1^N)}$, such that
$$P_\beta(u-\Id_{N})P_\beta=\bigsqcup_{\tilde\mu\in \cA_u} B_\cN e_{\tilde\mu} B_\cN.$$

If $x$ is a matrix with row/column set $\cS$, then let $\Res_\cR(x)$ denote the submatrix obtained by using only the rows and columns in the subset $\cR\subseteq \cS$.  The remainder of the proof shows that
\begin{equation}\label{SuperclassCombinatoricsFunction}
\begin{array}{r@{\ }c@{\ }c@{\ }c} \mathrm{rk}: & \left\{\begin{array}{@{}c@{}} \text{$P_\beta$-superclasses}\\ \text{of $\UT_{(\beta,\cP)}$}\end{array} \right\} & \longrightarrow & \cT_\cP^\beta\\
& P_\beta(u-\Id_{N})P_\beta &\longmapsto &\begin{array}{r@{}c@{\ }c@{\ }c} \mathrm{rk}_{\cA_u}: & F_\cP & \rightarrow & \ZZ_{\geq 0}\\ &(i,j) & \mapsto & \rank\Big(\Res_{\cQ_i\cup\cQ_j}(e_{\tilde\mu})\Big)\end{array}\end{array}
\end{equation}
gives a well-defined (does not depend on the choice of $\tilde\mu\in \cA_u$), bijective function.

To see well-defined, suppose $\tilde\mu,\tilde\nu\in \cA_u$.  Since $\UT_\beta\subseteq \UT_\cN$, we have that $e_{\tilde\mu}= ae_{\tilde\nu} b$ for some $a,b\in L_\beta=\GL_{\beta_1}\times \GL_{\beta_2}\times \cdots \times \GL_{\beta_\ell}$.  Thus,
\begin{align*}
\rank\Big(\Res_{\cQ_i\cup \cQ_j}(e_{\tilde\mu})\Big)&=\rank\Big(\Res_{\cQ_i\cup \cQ_j}(a e_{\tilde\nu} b)\Big)\\
&=\rank\Big(\Res_{\cQ_i\cup \cQ_i}( a)\Res_{\cQ_i\cup \cQ_j}(e_{\tilde\nu})\Res_{\cQ_j\cup \cQ_j}(b)\Big)\\
&=\rank\Big(\Res_{\cQ_i\cup \cQ_j}(e_{\tilde\nu})\Big).
\end{align*}

To see injectivity, fix $\tilde\nu\in \cT_\cP^{(1^N)}$, and let $\mu\in \cT_\cP^\beta$ be given by
$$\mu_{ij}=\rank\Big(\Res_{\cQ_i\cup \cQ_j}(e_{\tilde\nu})\Big).$$
Define $e_\mu\in \fkut_{(\beta,\cP)}$ by
\begin{equation}\label{emuChoice}
\Res_{\cQ_i\cup \cQ_k}(e_\mu)=\left[\begin{array}{c|c|c|c|c} 
0\cdot \Id_{\beta_i- \sum_{i<j\leq k} \mu_{ij}}& 0&  0 & 0 &  0 \\ \hline
0  & 0 & 0 &   w_{\mu_{ik}} & 0 \\ \hline
0 & 0 & 0   & 0 & 0 \\ \hline
0 & 0 & 0 & 0 & 0 \\ \hline
0 & 0 & 0 & 0 & 0\cdot \Id_{\beta_k-\sum_{i\leq j< k} \mu_{jk}} \end{array}\right].
\end{equation}
where 
$$w_n=\left[\begin{array}{ccc} 0 && 1\\ & \adots &\\ 1 & & 0\end{array}\right]\in \GL_n.$$
Then
 $$\rank\Big(\Res_{\cQ_i\cup \cQ_j}(e_{\mu}) \Big)=\rank\Big(\Res_{\cQ_i\cup \cQ_j}(e_{\tilde\nu})\Big)$$
 for all $i\prec_\cP j$.
 
Since each set $\cQ_i$ of rows and $\cQ_k$ of columns has the same number of ones for $e_\mu$ and $e_{\tilde\nu}$, there exist permutation matrices $l_k,r_k\in \GL_{\cQ_k}(\FF_q)$ such that 
$$e_\mu=\left(\begin{array}{c|c|c} l_1 &  & 0\\ \hline  & \ddots & \\ \hline 0 & & l_\ell\end{array}\right)  e_{\tilde\nu}
\left(\begin{array}{c|c|c} r_1 &  & 0\\ \hline  & \ddots & \\ \hline 0 & & r_\ell\end{array}\right).$$
Thus, $e_{\tilde\mu}$ and $e_\nu$ are in fact in the same superclass, and each superclass has a distinguished element $e_\mu$.
\end{proof}

\begin{remarks} \hfill

\begin{enumerate}
\item[(R1)] The construction (\ref{emuChoice}) gives us a superclass representative $e_\mu\in\fkut_{(\beta,\cP)}$ for each superclass.  If $\tilde\mu\in \cA_u\subseteq \cT_\cP^{(1^N)}$ is the element such that $e_\mu=e_{\tilde\mu}$, then $\tilde\mu$ is the unique element $\cA_u$ minimal first with respect to 
$$\sum_{i\prec_\cN j\prec_\cN  k\prec_\cN l}\tilde\mu_{ik}\tilde\mu_{jl}$$
and then with respect to 
$$\#\{i\prec_\cN j\prec_\cN k\mid \tilde\mu_{ik}=1\}.$$ 
\item[(R2)] By Theorem \ref{SuperclassIndexing}, the set $\cT_\cP^\beta$ also indexes the $P_\beta$-supercharacters of $\UT_\cP$.  Thus, if $y\in \fkut_{(\beta,\cP)}^*$ is in the orbit corresponding to $\lambda\in \cT_\cP^\beta$, then we will write
$$\chi^\lambda_\beta=\chi^{y}_\beta.$$
For the purpose of this paper it will not be necessary to fix a specific representative $e^*_{\lambda}\in \fkut_{(\beta,\cP)}^*$ of the orbit corresponding to $\lambda$.
\end{enumerate}
\end{remarks}

\section{Supercharacter formulas} \label{SectionSupercharacterFormulas}

This section works out character formulas for all the supercharacter theories described above.  Fundamentally, it involves weaving together two basic families of examples:
\begin{enumerate}
\item[(E1)] $B_\cN$-supercharacter theories of $\UT_\cP$ for $\cP\triangleleft \cN$.
\item[(E2)] The case where the unipotent polytope $(\beta,\cP)$ is a line segment (or where $\beta$ has exactly two parts).
\end{enumerate}
We first introduce some combinatorial statistics that will appear throughout the formulas, and then prove a result that shows how to compare $P_\beta$-supercharacter theories between different group (but for the same $\beta$).  Then we show how to compute supercharacter values for examples (E1) and (E2).  The main result then follows fairly quickly.

\subsection{Representation theoretic statistics}

There are a number of statistics that arise naturally in the $P_\beta$-supercharacter theories.  They naturally generalize their set partition analogues in the $P_{(1^N)}$-supercharacter theory of $\UT_\cN$ (see \cite{CPKR} for a more general algebraic framework for these statistics).   

For $\lambda\in \cT_\cP^\beta$ with $\cQ=\bdry^{-1}(\beta)$, there are a number of ways to measure the ``size" of a $\lambda$.  For example,
$$|\lambda| = \sum_{i\prec_\cP j} \lambda_{ij}$$
measures the lattice distance to the origin of the lattice point in the unipotent polytope.  However, geometric interpretations of the other statistics are unknown (at least to me).  Having more to do with the dimension of the corresponding modules, 
$$
\dim_L(\lambda) = \sum_{i\prec_\cP j\prec_\cQ k}\lambda_{ik}\beta_j\qquad \text{and}\qquad
\dim_R(\lambda) = \sum_{i\prec_\cQ j\prec_\cP k}\lambda_{ik}\beta_j
$$
give the \emph{left} and \emph{right dimensions} of $\lambda$ (respectively).  Note that if $\cP=\cQ$, then $\dim_R(\lambda)=\dim_L(\lambda)$.    To account for over-counting, we also require the \emph{crossing number}
$$\crs(\lambda)=\sum_{i\prec_\cQ j\prec_\cP k\prec_\cQ l}\lambda_{ik}\lambda_{jl}$$
of $\lambda$.  Lastly, if $\mu\in \cT_\cP^\beta$, the \emph{nestings} of $\mu$ in $\lambda$ are
$$\nst^\lambda_\mu = \sum_{i\prec_\cP j \prec_{\cP} k\prec_{\cP}l}   \lambda_{il}\mu_{jk}.$$

\begin{example} if 
$\beta=(3,6,3,4,5,1)$,
$$\lambda=
\begin{tikzpicture}[scale=.4,baseline=.5cm]
\foreach \x/\y/\z in {-2/2/3,-1/1/6,0/0/3,1/-1/4,2/-2/5,3/-3/1}
		{\fill[gray!20!white] (\x,\y)+(-.5,-.5) rectangle ++(.5,.5);
		\node at (\x,\y) {$\scs\z$};}
\foreach \x/\y/\num in 
	{0/2/2,1/2/0,2/2/1,3/2/0,
	0/1/0,1/1/0,2/1/1,3/1/1,
		1/0/0,2/0/1,3/0/0}
{
  \draw (\x,\y) +(-0.5,-0.5) rectangle ++(0.5,0.5);
  \draw (\x,\y) node {$\scs\num$};
}
\end{tikzpicture}
\qquad\text{and}\qquad
\mu=
\begin{tikzpicture}[scale=.4,baseline=.5cm]
	\foreach \x/\y/\z in {-2/2/3,-1/1/6,0/0/3,1/-1/4,2/-2/5,3/-3/1}
		{\fill[gray!20!white] (\x,\y)+(-.5,-.5) rectangle ++(.5,.5);
		\node at (\x,\y) {$\scs\z$};}
	\foreach \x/\y/\num in 
	{0/2/0,1/2/1,2/2/0,3/2/1,
	0/1/1,1/1/0,2/1/2,3/1/0,
		1/0/1,2/0/2,3/0/0}
{
  \draw (\x,\y) +(-0.5,-0.5) rectangle ++(0.5,0.5);
  \draw (\x,\y) node {$\scs\num$};
}
\end{tikzpicture}
$$
then 
\begin{align*}
|\lambda| &= 6\cdot 0+ 4\cdot 1 + 1\cdot 2\\
\dim_L(\lambda) &=\begin{array}{c@{\ }c@{\ }c@{\ }c@{\ }c@{\ }c@{\ }c}
1\cdot (3+4) &+& 1\cdot (3+4) &+&1\cdot 4 &+& 1\cdot (3+4+5)\\
\begin{tikzpicture}[scale=.4,baseline=.5cm]
		\fill[gray!30!white] (0,1)+(-.3,-1.5) rectangle ++(.3,1.5);
		\fill[gray!30!white] (1,.5)+(-.3,-2) rectangle ++(.3,2);
		\node at (0,0) {$\scs 3$};
		\node at (1,-1) {$\scs 4$};	
\draw  (2,2) + (-0.4,-0.4) rectangle ++(0.4,0.4);
\fill[gray!50!white] (.5,2) + (-1,-.2) rectangle ++ (1,.2);
\foreach \x/\y/\num in 
	{0/2/2,1/2/0,2/2/1,3/2/0,
	0/1/0,1/1/0,2/1/1,3/1/1,
		1/0/0,2/0/1,3/0/0}
{
  \draw (\x,\y) +(-0.5,-0.5) rectangle ++(0.5,0.5);
  \draw (\x,\y) node {$\scs\num$};
}
\end{tikzpicture} 
& & 
\begin{tikzpicture}[scale=.4,baseline=.5cm]
		\fill[gray!30!white] (0,.5)+(-.3,-1) rectangle ++(.3,1);
		\fill[gray!30!white] (1,0)+(-.3,-1.5) rectangle ++(.3,1.5);
		\node at (0,0) {$\scs 3$};
		\node at (1,-1) {$\scs 4$};	
	\draw  (2,1) + (-0.4,-0.4) rectangle ++(0.4,0.4);
\fill[gray!50!white] (.5,1) + (-1,-.2) rectangle ++ (1,.2);	
\foreach \x/\y/\num in 
	{0/2/2,1/2/0,2/2/1,3/2/0,
	0/1/0,1/1/0,2/1/1,3/1/1,
		1/0/0,2/0/1,3/0/0}
{
  \draw (\x,\y) +(-0.5,-0.5) rectangle ++(0.5,0.5);
  \draw (\x,\y) node {$\scs\num$};
}
\end{tikzpicture}
&& 
\begin{tikzpicture}[scale=.4,baseline=.5cm]
		\fill[gray!30!white] (1,-.5)+(-.3,-1) rectangle ++(.3,1);
		\node at (1,-1) {$\scs 4$};	
\draw  (2,0) + (-0.4,-0.4) rectangle ++(0.4,0.4);
\fill[gray!50!white] (1,0) + (-.5,-.2) rectangle ++ (.5,.2);	
\foreach \x/\y/\num in 
	{0/2/2,1/2/0,2/2/1,3/2/0,
	0/1/0,1/1/0,2/1/1,3/1/1,
		1/0/0,2/0/1,3/0/0}
{
  \draw (\x,\y) +(-0.5,-0.5) rectangle ++(0.5,0.5);
  \draw (\x,\y) node {$\scs\num$};
}
\end{tikzpicture}
&  &
\begin{tikzpicture}[scale=.4,baseline=.5cm]
		\fill[gray!30!white] (0,.5)+(-.3,-1) rectangle ++(.3,1);
		\fill[gray!30!white] (1,0)+(-.3,-1.5) rectangle ++(.3,1.5);
		\fill[gray!30!white] (2,-.5)+(-.3,-2) rectangle ++(.3,2);
		\node at (0,0) {$\scs 3$};
		\node at (1,-1) {$\scs 4$};	
		\node at (2,-2) {$\scs 5$};
	\draw  (3,1) + (-0.4,-0.4) rectangle ++(0.4,0.4);
\fill[gray!50!white] (1,1) + (-1.5,-.2) rectangle ++ (1.5,.2);	
\foreach \x/\y/\num in 
	{0/2/2,1/2/0,2/2/1,3/2/0,
	0/1/0,1/1/0,2/1/1,3/1/1,
		1/0/0,2/0/1,3/0/0}
{
  \draw (\x,\y) +(-0.5,-0.5) rectangle ++(0.5,0.5);
  \draw (\x,\y) node {$\scs\num$};
}
\end{tikzpicture}\end{array}\\
\dim_R(\lambda)&=2\cdot 6 + 1\cdot (6+3)+1\cdot 3+1\cdot 3\\
\nst^\lambda_\mu &=1\cdot 1+ 3\cdot 1\cdot 1 + 2\cdot 1\\
\crs(\lambda)&= \begin{array}{c@{\ }c@{\ }c@{\ }c@{\ }c}
2\cdot 1&+&2\cdot 1&+&1\cdot 1\\
\begin{tikzpicture}[scale=.4,baseline=.5cm]
		\fill[gray!50!white] (0,1)+(-.2,-.5) rectangle ++(.2,.5);
		\fill[gray!50!white] (.5,1) + (-1,-.2) rectangle ++ (1,.2);	
\draw  (0,2) + (-0.4,-0.4) rectangle ++(0.4,0.4);
\draw  (2,1) + (-0.4,-0.4) rectangle ++(0.4,0.4);
\foreach \x/\y/\num in 
	{0/2/2,1/2/0,2/2/1,3/2/0,
	0/1/0,1/1/0,2/1/1,3/1/1,
		1/0/0,2/0/1,3/0/0}
{
  \draw (\x,\y) +(-0.5,-0.5) rectangle ++(0.5,0.5);
  \draw (\x,\y) node {$\scs\num$};
}
\end{tikzpicture} 
& & 
\begin{tikzpicture}[scale=.4,baseline=.5cm]
		\fill[gray!50!white] (0,1)+(-.2,-.5) rectangle ++(.2,.5);
		\fill[gray!50!white] (1,1) + (-1.5,-.2) rectangle ++ (1.5,.2);	
\draw  (0,2) + (-0.4,-0.4) rectangle ++(0.4,0.4);
\draw  (3,1) + (-0.4,-0.4) rectangle ++(0.4,0.4);
\foreach \x/\y/\num in 
	{0/2/2,1/2/0,2/2/1,3/2/0,
	0/1/0,1/1/0,2/1/1,3/1/1,
		1/0/0,2/0/1,3/0/0}
{
  \draw (\x,\y) +(-0.5,-0.5) rectangle ++(0.5,0.5);
  \draw (\x,\y) node {$\scs\num$};
}
\end{tikzpicture} 
&& 
\begin{tikzpicture}[scale=.4,baseline=.5cm]
		\fill[gray!50!white] (2,.5)+(-.2,-1) rectangle ++(.2,1);
		\fill[gray!50!white] (1,1) + (-1.5,-.2) rectangle ++ (1.5,.2);	
\draw  (2,2) + (-0.4,-0.4) rectangle ++(0.4,0.4);
\draw  (3,1) + (-0.4,-0.4) rectangle ++(0.4,0.4);
\foreach \x/\y/\num in 
	{0/2/2,1/2/0,2/2/1,3/2/0,
	0/1/0,1/1/0,2/1/1,3/1/1,
		1/0/0,2/0/1,3/0/0}
{
  \draw (\x,\y) +(-0.5,-0.5) rectangle ++(0.5,0.5);
  \draw (\x,\y) node {$\scs\num$};
}
\end{tikzpicture} 
\end{array}
\end{align*}

\end{example}

The following lemma gives an algebraic foundation for most of these statistics.   If one uses the standard representatives coming from the $B_\cN$-supercharacter theory the proof is relatively straight-forward, and the details are left to the reader.
\begin{lemma}\label{StatsInterpretation}
Let $\lambda\in \cT_\cP^\beta$ with $\tilde\lambda\in \cT_{\cP}^{(1^{N})}$ such that $e^*_{\tilde\lambda}\in \fkut_{(\beta,\cP)}^*$ is in the $\lambda$-orbit.  Then
\begin{equation*}
|\lambda|=\rank(e_{\tilde\lambda}),\quad
q^{\dim_L(\lambda)}=|\UT_\beta e^*_{\tilde\lambda}|, \quad q^{\dim_R(\lambda)}= |e^*_{\tilde\lambda}\UT_\beta|\quad \text{and}\quad
q^{\crs(\lambda)}=|\UT_\beta e^*_{\tilde\lambda}\cap e^*_{\tilde\lambda}\UT_\beta|.
\end{equation*}
\end{lemma}

\subsection{Comparing $P_\beta$-supercharacter theories}

Fix a unipotent polytope $(\beta,\cP)$ and let $\cQ=\bdry^{-1}(\beta)$.   Then we have an injective function $\Ext_{\cP}^\cQ:  \cT_\cP^\beta   \longrightarrow  \cT_\cQ^\beta$ given by
$$\Ext_{\cP}^\cQ(\lambda)_{ij}=\left\{\begin{array}{@{}ll@{}} \lambda_{ij} & \text{if $i\prec_{\fat_\beta(\cP)}j$,}\\ 0 & \text{otherwise.}\end{array}\right.$$ 
This gives a way to compare supercharacter values between the two theories. The following proposition shows that the representation theory mirrors the combinatorics as well as can be expected.

\begin{proposition}\label{RestrictingSupercharacters} Let $(\beta,\cP)$ be a unipotent polytope with $\cQ=\bdry^{-1}(\beta)$.  
For $\lambda\in \cT_\cP^\beta$,
$$\frac{\chi^\lambda_{\beta}}{\chi_\beta^\lambda(1)}=\frac{\Res^{\UT_\beta}_{\UT_{(\beta,\cP)}}(\chi_\beta^{\Ext_{\cP}^\cQ(\lambda)})}{\chi_\beta^{\Ext_{\cP}^\cQ(\lambda)}(1)}.$$
\end{proposition}
\begin{proof} 
Fix $e_{\tilde\lambda}^*\in \fkut_{(\beta,\cP)}^*$  in the orbit corresponding to $\lambda$.  Then since $\fkut_{(\beta,\cP)}^*$ is invariant under left and right multiplication by $P_\beta$, we also have  $e_{\tilde\lambda}^*\in \fkut_\cQ^*$ is in the orbit corresponding to $\Ext_{\cP}^\cQ(\lambda)$.   Let  $u-\Id_{|\beta|}\in \fkut_{(\beta,\cP)}\subseteq \fkut_\cQ$.  Then by definition (\ref{SupercharacterDefinition}) and the invariance of $\fkut_{(\beta,\cP)}^*$ under $P_\beta$,
\begin{equation*}
\frac{\Res^{\UT_\cQ}_{\UT_{(\beta,\cP)}}(\chi_\beta^{\Ext_{\cP}^\cQ(\lambda)})(u)}{\chi_\beta^{\Ext_{\cP}^\cQ(\lambda)}(1)}=\frac{1}{|P_\beta e_{\tilde\lambda}^* P_\beta|}\sum_{e_{\tilde\nu}^*\in P_\beta e_{\tilde\lambda}^* P_\beta} \vartheta\circ e_{\tilde\nu}^*(u-\Id_{|\beta|})
=\frac{\chi_\beta^\lambda(u)}{\chi_\beta^\lambda(1)},
\end{equation*}
as desired.
\end{proof}

\subsection{Example: $B_\cN$-supercharacter theories}

In the case that $\cQ=\cN$, then all normal pattern subgroups are also normal Levi compatible subgroups.  Here the $B_\cN$-supercharacter formula for $\UT_{(\beta,\cP)}$ are obtained by restricting the $B_\cN$-supercharacter formulas for $\UT_\cN$.  

Up to scaling, the following character formula appears in \cite{BT}, but does not have an explicit published proof.  For $\lambda,\mu\in\cT_\cN^{(1^N)}$, let $\lambda\cap \mu \in \cT_\cN^{(1^N)}$ be given by
$$(\lambda\cap \mu)_{ij}=\lambda_{ij}\mu_{ij}.$$
\begin{proposition}  \label{BasicBFormula} For $\lambda,\mu\in \cT_{\cN}^{(1^{N})}$, 
$$\chi_{(1^N)}^\lambda(u_\mu)=\left\{\begin{array}{@{}ll@{}} 
\frac{q^{\dim_L(\lambda)+\dim_R(\lambda)}(q-1)^{|\lambda|}}{q^{\crs(\lambda)}}\frac{1}{q^{\nst_\mu^\lambda}(1-q)^{|\lambda\cap \mu|}}, & \text{if $\lambda_{ik}\mu_{ij}=\lambda_{ik}\mu_{jk}=0$ for $i\prec_\cN j\prec_{\cN} k$},\\ 0 & \text{otherwise.}
\end{array}\right.$$
\end{proposition}

From Proposition \ref{RestrictingSupercharacters} we get the following corollary.

\begin{corollary}\label{BSupercharacterFormula}
 For $\cP\triangleleft\cN$ and $\lambda,\mu \in \cT_\cP^{(1^{N})}$,
$$\chi_{(1^N)}^\lambda(u_\mu)=\left\{\begin{array}{@{}ll@{}} 
\frac{q^{\dim_L(\lambda)+\dim_R(\lambda)}(q-1)^{|\lambda|}}{q^{\crs(\lambda)}}\frac{1}{q^{\nst_\mu^\lambda}(1-q)^{|\lambda\cap \mu|}}, & \text{if $\lambda_{ik}\mu_{ij} = \lambda_{ik}\mu_{jk}=0$ for $i\prec_\cN j\prec_{\cN} k$},\\ 0 & \text{otherwise.}
\end{array}\right.$$
\end{corollary} 

\begin{proof}
By  Proposition \ref{BasicBFormula},
\begin{equation*}
\chi_{(1^N)}^{\Ext_{\cP}^{\cN}(\lambda)}(1)=\frac{q^{\dim_L(\Ext_{\cP}^{\cN}(\lambda))+\dim_R(\Ext_{\cP}^{\cN}(\lambda))}(q-1)^{|\lambda|}}{q^{\crs(\Ext_{\cP}^{\cN}(\lambda))}}.
\end{equation*}
Let $e^*_\lambda\in \fkut_{((1^N),\cP)}^*$ be in the $\lambda$-orbit.  Then by (\ref{SupercharacterDefinition}),
\begin{align*}\chi_{(1^N)}^\lambda(1)&=|B_\cN e^*_\lambda B_\cN|\\
&=(q-1)^{|\lambda|}|\UT_\cN e^*_\lambda \UT_\cN|\\
&=(q-1)^{|\lambda|}\frac{|\UT_\cN e^*_\lambda||e^*_\lambda \UT_\cN|}{|\UT_\cN e^*_\lambda \cap e^*_\lambda \UT_\cN|}\\
&=(q-1)^{|\lambda|}\frac{q^{\dim_L(\lambda)+\dim_R(\lambda)}}{q^{\crs(\lambda)}},
\end{align*}
where the last equality follows from Lemma \ref{StatsInterpretation}.  Apply Proposition \ref{RestrictingSupercharacters} to Proposition \ref{BasicBFormula} to obtain the formula.
\end{proof}

\subsection{The basic building block: the line.}

Consider the case where 
$$\beta=(m,n), \quad\cQ=
\bdry^{-1}(\beta)=\begin{tikzpicture}[baseline=.4cm, scale=.8]
	\foreach \x/\y/\z in {{.5}/0/1,{1.5}/0/2,{3.5}/0/m}
		{\node (\x) at (\x,\y) [inner sep = -1pt] {$\scs \bullet$};
		\node at (\x,-.3) {$\scs \z$};
		}
	\foreach \x/\y/\z in {0/1/{m+1},1/1/{m+2},4/1/{m+n}}
		{\node (\x) at (\x,\y) [inner sep = -1pt] {$\scs \bullet$};
		\node at (\x,1.3) {$\scs \z$};
		}
	\foreach \x/\y in {2.5/0,2.5/1}
		\node at (\x,\y) {$\cdots$};
	\foreach \x in {{.5},{1.5},{3.5}}
		{\foreach \y in {0,1,4}
		\draw (\x,0) -- (\y,1);
		}
\end{tikzpicture}
\qquad \text{where}\qquad \cN=\begin{tikzpicture}[scale=.7,baseline=1cm]
	\foreach \y/\z in {0/1,1/2,3/{m+n}}
		{\node (\y) at (0,\y) [inner sep = -1pt] {$\scs\bullet$};
		\node at (.6,\y) {$\scs\z$};}
		\draw (0) -- (1) -- (0,1.5);
		\draw (0,2.5) -- (0,3);
		\node at (0,2.1) {$\vdots$};
\end{tikzpicture}.$$

Then
$$\UT_{(\beta,\cQ)}=\left\{\left[\begin{array}{c|c} \Id_m & A\\ \hline 0 & \Id_n\end{array}\right]\ \bigg|\ A\in M_{m\times n}(\FF_q)\right\}\cong (\FF_q^+)^{mn}.$$
In this case, the indexing set for the supercharacters and superclasses is given by
$$\cT_{\begin{tikzpicture}[scale=.3]
\foreach \y in {0,1}
	\node (\y) at (0,\y) [inner sep=-1pt] {$\scs\bullet$};
	\draw (0) -- (1);
	\end{tikzpicture}}^{(m,n)} =\{0,1,\ldots,\min\{m,n\}\}, \quad\text{where the superclass of } \left[\begin{array}{c|c} \Id_m & A\\ \hline 0 & \Id_n\end{array}\right]\text{ is labelled by } \rank(A).$$

\begin{theorem}\label{1BlockSupercharacters}
 If  $0\leq j,l\leq \min\{m,n\}$ and $u_{(j)}=\Id_{m+n}+e_{(j)}\in \UT_{(\beta,\cQ)}$,  then
$$\chi^{(l)}_{(m,n)} (u_{(0)})=|\GL_l(\FF_q)|\qbin{m}{l}\qbin{n}{l}=\begin{array}{c}\text{$\#$  $m\times n$ matrices}\\ \text{of rank $l$}\end{array},$$
and
$$\chi_{(m,n)}^{(l)}(u_{(j)})=\sum_{(a,b)\vDash l} (-1)^a q^{bj+\binom{a}{2}}\qbin{j}{a}\chi^{(b)}_{(m-j, n-j)} (u_{(0)}).$$
\end{theorem}
\begin{remark} 
Note that in the sum not all compositions $(a,b)$ of $l$ give nonzero terms.   We use the convention that $|\GL_0(\FF_q)|=1$.
\end{remark}
\begin{proof}
Fix $e_{(l)}^*\in \fkut_\cQ^*$ in the orbit corresponding to $l$, and let
$$u_{(j)}=\left[\begin{array}{c|c} \Id_m & \begin{array}{@{}c|c@{}} 0 & 0 \\ \hline w_j & 0\end{array} \\ \hline 0 & \Id_n\end{array}\right], \qquad \text{where} \qquad w_j=\left[\begin{array}{ccc} 0 & & 1\\ & \adots & \\ 1 & & 0 \end{array}\right]\in \GL_j,$$
be the usual representative (as in (\ref{emuChoice})).
By (\ref{BSupercharacterDecomposition}),
$$\chi_{(m,n)}^{(l)}(u_{(j)})=\sum_{B_\cN e_{\tilde\nu}^* B_\cN\in P_Q e_{(l)}^* P_\beta} \chi_{(1^N)}^{e_{\tilde\nu}^*}(u_{(j)})=\sum_{\nu\in \cT_\cQ^{(1^N)}\atop |\nu|=l} \chi_{(1^N)}^{\nu}(u_{(j)}).$$
Let 
\begin{align*}
\cA&=\{(m-j+i,m+j+1-i)\mid 1\leq i\leq j\}\\
\cB&=\{1,\ldots, m-j\}\times \{m+j+1,\ldots,m+n\}
\end{align*}
or visually, the coordinates
$$\left[\begin{array}{c|c|c|c}
 \Id_{m-j} & 0 & 0 & \begin{array}{ccc} \cB & \cdots & \cB\\ 
 \vdots & \ddots & \vdots \\ \cB & \cdots & \cB\end{array}\\ \hline
0 & \Id_j & \begin{array}{ccc} 0 & & \cA\\ & \adots & \\ \cA & & 0\end{array} & 0 \\ \hline
0  & 0 & \Id_j & 0\\ \hline
0 & 0 & 0 & \Id_{n-j}
\end{array}\right].$$
Note that $\chi_{(1^N)}^\nu(u_{(j)})=0$ unless
$$\nu_{ik}\neq 0 \qquad\text{implies $(i,k)\in \cA\cup \cB.$}$$
For $\nu\in \cT_\cQ^{(1^N)}$, let $\nu_\cA,\nu_\cB\in \cT_\cQ^{(1^N)}$ be given by
$$(\nu_\cA)_{ij}=\left\{\begin{array}{ll} \nu_{ij} & \text{if $(i,j)\in \cA$,}\\ 0 & \text{otherwise,}\end{array}\right.\qquad
(\nu_\cB)_{ij}=\left\{\begin{array}{ll} \nu_{ij} & \text{if $(i,j)\in \cB$,}\\ 0 & \text{otherwise.}\end{array}\right.$$
Then by Corollary \ref{BSupercharacterFormula},
\begin{align}
\chi_{(m,n)}^{(l)}(u_{(j)})&=\sum_{\nu=\nu_\cA+\nu_\cB\in \cT_\cQ^{(1^N)}\atop |\nu| =l}\frac{q^{\dim_L(\nu)+\dim_R(\nu)}(q-1)^{l}}{q^{\crs(\nu_\cB)}}\frac{1}{q^{\nst_{\cA}^{\nu_\cA}+|\nu_\cB|j}}\left(\frac{1}{1-q}\right)^{|\nu_\cA|} \label{KeyIdentity}\\
&=(q-1)^l\sum_{\alpha,\gamma\in \cT_\cQ^{(1^N)}\atop \alpha_\cA=\alpha, \gamma_\cB=\gamma, |\alpha+\gamma| =l}\frac{q^{\dim_L(\gamma)+\dim_R(\gamma)}}{q^{\crs(\gamma)+|\gamma| j}}\frac{q^{\dim_L(\alpha)+\dim_R(\alpha)}}{q^{\nst_{\cA}^{\alpha}}}\left(\frac{1}{1-q}\right)^{|\alpha|}\notag \\
&=(q-1)^l\sum_{{(a,b)\vDash l\atop a\leq j} \atop b\leq m-j}\left(\frac{1}{1-q}\right)^{a}\sum_{{\alpha,\gamma\in \cT_\cQ^{(1^N)}\atop \alpha_\cA=\alpha, \gamma_\cB=\gamma,}\atop |\alpha|=a,|\gamma|=b}\frac{q^{\dim_L(\gamma)+\dim_R(\gamma)}}{q^{\crs(\gamma)+b j}}\frac{q^{\dim_L(\alpha)+\dim_R(\alpha)}}{q^{\nst_{\cA}^{\alpha}}}.\notag
\end{align}
Note that choice of $\gamma$ with $\gamma_\cB=\gamma$ and $|\gamma|=b$ is determined by a triple $(\cR,\cC,w)$ where $\cR\times \cC\subseteq \cB$ with $|\cR|=|\cC|=b$ and $w:\cR\rightarrow \cC$ is a bijection (then  $\gamma_{ik}=1$ if and only if $i\in \cR$ and $k=w(i)$).  For such a triple $(\cR,\cC,w)$ corresponding to $\gamma$ we have
$$\frac{q^{\dim(\gamma)}}{q^{bj}}=q^{\wt_{[1,m-j]}^\uparrow(\cR)+\wt_{[m+j+1,m+n]}^\downarrow(\cC)+bj},\qquad \text{and}\qquad \crs(\gamma)=\mathrm{inv}(w\circ w_b).$$
A choice of $\alpha$ with $\alpha_\cA=\alpha$ and $|\alpha|=a$ is completely determined by $\cD\subseteq \{m+1,\ldots, m+j\}$ with $|\cD|=a$.  In this case, 
$$\frac{q^{\dim(\alpha)}}{q^{\nst_{\cA}^\alpha}}=q^{\wt_{[m+1,m+j]}^\downarrow(\cD)}.$$
For fixed $a$, $b$, $\cC$ and $\cR$ we can sum over choices of $w$, $\cD$ independently.   Then (\ref{InvolutionSum}) and (\ref{WeightqBinomial}) give
\begin{equation*}
\chi_{(m,n)}^{(l)}(u_{(j)})=(q-1)^l\sum_{{(a,b)\vDash l\atop a\leq j} \atop b\leq m-j}q^{\binom{a}{2}}\qbin{j}{a}\left(\frac{1}{1-q}\right)^{a}\sum_{\cR\times \cC\subseteq \cB\atop |\cR|=|\cC|=b} q^{\wt_{[1,m-j]}^\uparrow(\cR)+\wt_{[m+j+1,m+n]}^\downarrow(\cC)+bj} \frac{[b]!}{q^{\binom{b}{2}}}.
\end{equation*}
Finally, sum over $\cC$ and $\cR$ independently and apply (\ref{WeightqBinomial})  to get
\begin{align*}
\chi_{(m,n)}^{(l)}(u_{(j)}) &=(q-1)^l\sum_{{(a,b)\vDash l\atop a\leq j} \atop b\leq m-j}q^{\binom{a}{2}}\qbin{j}{a}\left(\frac{1}{1-q}\right)^{a} q^{bj+2\binom{b}{2}}\qbin{m-j}{b}\qbin{n-j}{b} \frac{[b]!}{q^{\binom{b}{2}}}\\
&=\sum_{(a,b)\vDash l}(-1)^aq^{bj+\binom{a}{2}}\qbin{j}{a} |\GL_b(\FF_q)|\qbin{m-j}{b}\qbin{n-j}{b},
\end{align*}
as desired.
\end{proof}

\subsection{General supercharacter formula}

Let $(\beta,\cP)$ be a unipotent polytope with $\sInt(\cP)$ as in Section \ref{SectionNormalPosets}.  For $\lambda,\mu\in \cT_\cP^\beta$, let
$$\begin{array}{r@{\ }c@{\ }c@{\ }c}\loc_\mu^\lambda: & \sInt(\cP) & \longrightarrow  & \ZZ_{\geq 0}\times \ZZ_{\geq 0}\\
&(j,l) & \mapsto  &\dd\Big(\beta_j-\sum_{j\prec_\cP k\prec_\cP l} \mu_{jk}-\sum_{l\prec_\cP m} \lambda_{jm},\beta_{l}-\sum_{j\prec_\cP k\prec_\cP l} \mu_{kl}-\sum_{i\prec_\cP j} \lambda_{il}\Big).
\end{array}$$
For example, if $\beta=(3,6,3,4,5,1)$,
$$\lambda=
\begin{tikzpicture}[scale=.4,baseline=.5cm]
	\foreach \x/\y/\num in 
	{0/2/2,1/2/0,2/2/1,3/2/0,
	0/1/0,1/1/0,2/1/1,3/1/1,
		1/0/0,2/0/1,3/0/0}
{
  \draw (\x,\y) +(-0.5,-0.5) rectangle ++(0.5,0.5);
  \draw (\x,\y) node {$\scs\num$};
}
\end{tikzpicture}
\qquad\text{and}\qquad
\mu=
\begin{tikzpicture}[scale=.4,baseline=.5cm]
	\foreach \x/\y/\num in 
	{0/2/0,1/2/1,2/2/0,3/2/1,
	0/1/1,1/1/0,2/1/2,3/1/0,
		1/0/1,2/0/2,3/0/0}
{
  \draw (\x,\y) +(-0.5,-0.5) rectangle ++(0.5,0.5);
  \draw (\x,\y) node {$\scs\num$};
}
\end{tikzpicture}
$$
then 
\begin{align*}
\loc^\lambda_{\mu}(2,5)&=\bigg(\beta_2-
\begin{tikzpicture}[scale=.4,baseline=.4cm]
	\fill[gray!30!white] (0,1)+(-0.5,-0.5) rectangle ++(0.5,0.5);
	\fill[gray!30!white] (1,1)+(-0.5,-0.5) rectangle ++(0.5,0.5);
	\draw (2,1) +(-0.4,-0.4) rectangle ++(0.4,0.4);
	\foreach \x/\y/\num in 
	{0/2/0,1/2/1,2/2/0,3/2/1,
	0/1/1,1/1/0,2/1/2,3/1/0,
		1/0/1,2/0/2,3/0/0}
{
  \draw (\x,\y) +(-0.5,-0.5) rectangle ++(0.5,0.5);
  \draw (\x,\y) node {$\scs\num$};
}
\end{tikzpicture}\ -
\begin{tikzpicture}[scale=.4,baseline=.4cm]
	\fill[gray!30!white] (3,1)+(-0.5,-0.5) rectangle ++(0.5,0.5);
	\draw (2,1) +(-0.4,-0.4) rectangle ++(0.4,0.4);
	\foreach \x/\y/\num in 
	{0/2/2,1/2/0,2/2/1,3/2/0,
	0/1/0,1/1/0,2/1/1,3/1/1,
		1/0/0,2/0/1,3/0/0}
{
  \draw (\x,\y) +(-0.5,-0.5) rectangle ++(0.5,0.5);
  \draw (\x,\y) node {$\scs\num$};
}
\end{tikzpicture}
,\beta_5-
\begin{tikzpicture}[scale=.4,baseline=.4cm]
	\fill[gray!30!white] (2,0)+(-0.5,-0.5) rectangle ++(0.5,0.5);
	\draw (2,1) +(-0.4,-0.4) rectangle ++(0.4,0.4);
	\foreach \x/\y/\num in 
	{0/2/0,1/2/1,2/2/0,3/2/1,
	0/1/1,1/1/0,2/1/2,3/1/0,
		1/0/1,2/0/2,3/0/0}
{
  \draw (\x,\y) +(-0.5,-0.5) rectangle ++(0.5,0.5);
  \draw (\x,\y) node {$\scs\num$};
}
\end{tikzpicture}
-\begin{tikzpicture}[scale=.4,baseline=.4cm]
	\fill[gray!30!white] (2,2)+(-0.5,-0.5) rectangle ++(0.5,0.5);
	\draw (2,1) +(-0.4,-0.4) rectangle ++(0.4,0.4);
	\foreach \x/\y/\num in 
	{0/2/2,1/2/0,2/2/1,3/2/0,
	0/1/0,1/1/0,2/1/1,3/1/1,
		1/0/0,2/0/1,3/0/0}
{
  \draw (\x,\y) +(-0.5,-0.5) rectangle ++(0.5,0.5);
  \draw (\x,\y) node {$\scs\num$};
}
\end{tikzpicture}
\bigg)\\
&=(6-(1+0)-1,5-2-1)\\
&=(4,2).
\end{align*}
\begin{theorem} \label{FullFormula}
For $\lambda,\mu\in \cT_\cP^\beta$ and $u_\mu=1+e_\mu$,
$$\chi^\lambda_\beta(u_\mu)=\frac{q^{\dim_L(\lambda)+\dim_R(\lambda)}}{q^{\nst^\lambda_\mu+\crs(\lambda)}} \prod_{j\prec_\cP l} \chi^{(\lambda_{jl})}_{\loc^\lambda_{\mu}(j,l)}(u_{(\mu_{jl})}).$$
\end{theorem}
\begin{remark}
The following proof is designed to directly construct the above ``factorization."  Alternatively, it is perhaps more straightforward (though with even more tedious book-keeping) to follow the proof of Theorem \ref{1BlockSupercharacters}, and then to retroactively observe the desired factorization.
\end{remark}
\begin{proof}
Let $\tilde\lambda=\Ext_{\cP}^{\cN}(\lambda)$.
Then by (\ref{BSupercharacterDecomposition}),
$$\chi_\beta^\lambda(u_\mu)=\sum_{\tilde\nu\in \cT_\cP^{(1^N)}\atop e^*_{\tilde\nu}\in P_\beta e^*_{\tilde\lambda} P_\beta} \chi_{(1^N)}^{\tilde\nu}(u_\mu).$$
Let $\cQ=\bdry^{-1}(\beta)$ with associated set composition $(\cQ_1,\ldots, \cQ_\ell)$.  
By a dual analogue to Theorem \ref{SuperclassIndexing}, 
  $e_{\tilde\nu}^*\in P_\beta e^*_{\tilde\lambda} P_\beta$ with $\tilde\nu\in \cT_\cP^{(1^N)}$ if and only if
$$\rank\Big(\Res_{\cQ_i\cup \cQ_j}(e_{\tilde\nu})\Big)=\rank\Big(\Res_{\cQ_i\cup \cQ_j}(e_{\tilde\lambda})\Big)=\lambda_{ij}\qquad \text{for all $i\prec_\cP j$}.$$
Let $e_\mu=e_{\tilde\mu}$ for $\tilde\mu\in \cT_\cP^{(1^N)}$.  By Corollary \ref{BSupercharacterFormula},
$$\chi_\beta^\lambda(u_\mu)=\sum_{\tilde\nu\in \cT_\cP^{(1^N)}\atop \left(\rank(\Res_{\cQ_i\cup \cQ_j} (e_{\tilde\nu}))\right)=\lambda}\delta_{\tilde\mu}^{\tilde\nu}\frac{q^{\dim_L(\tilde\nu)+\dim_R(\tilde\nu)}}{q^{\crs(\tilde\nu)+\nst_{\tilde\mu}^{\tilde\nu}}}(-1)^{|\tilde\nu\cap \tilde\mu|}(q-1)^{|\tilde\nu|-|\tilde\nu\cap \tilde\mu|},
$$
where
$$\delta_{\tilde\mu}^{\tilde\nu}=\prod_{i<j<k} \delta_{\tilde\nu_{ik}\tilde\mu_{ij},0}\delta_{\tilde\nu_{ik}\tilde\mu_{jk},0}$$
keeps track of which terms are zero (from Corollary \ref{BSupercharacterFormula}).

Note that for each set of choices of made for the blocks,
$$\{\gamma_{jl}\in \cT_\cP^{(1^N)} \mid 1\leq j\prec_\cP l<\ell, \Res_{\cQ_i\times \cQ_k}(\gamma_{jl})=\delta_{(i,k)(j,l)}\gamma_{jl}, \rank(\gamma_{jl})=\lambda_{jl}\},$$
such that no two $\gamma_{jl}$ have 1's in the same row or column, then
$$\sum_{j\prec_\cP l} \gamma_{jl}\in \cT_\cP^{(1^N)}.$$
Given a total order $\prec$ on the coordinates in $F_\cP$, we can iteratively construct an element $\tilde\nu\in \cT_\cP^{(1^N)}$ by choosing each $\cQ_j\times \cQ_l$ according to $\prec$.  We will use the $F_\cP$-total order induced by the $F_\cQ$-total order
$$(1,\ell)\prec (1,\ell-1)\prec (2,\ell)\prec (1,\ell-2)\prec \cdots \prec (\ell-1,\ell) \qquad \text{or}
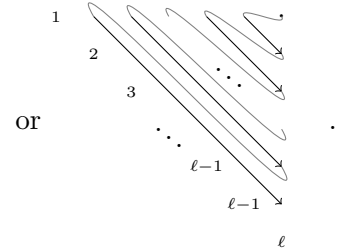
\begin{tikzpicture}[scale=.5,baseline=1.5cm]
\node at (6,6) {$\cdot$};
\node at (4.6,4.6) {$\scs\ddots$};
\foreach \x in {1,2,4,5}
	\draw[->] (\x,6) --  (6,\x);
\foreach \x in {2,3,4,5,6}
	\draw[gray] (6,\x) to [out=-60,in=120] (\x-1,6);
\foreach \x/\y/\z in {0/6/1,1/5/2,2/4/3,3/3/{\ddots},4/2/{\ell-1},5/1/{\ell-1},6/0/{\ell}}
	\node at (\x,\y) {$\scscs\z$};
\end{tikzpicture}.$$
Then we obtain a sequence
$$\{\tilde\nu^{(jl)}\in \cT_\cP^{(1^N)}\mid j\prec_\cP l\}, \quad \text{where} \quad \tilde\nu^{(jl)}= \gamma_{jl}+ \sum_{(i,k)\prec (j,l) } \gamma_{ik}.$$

At each step there will be limits on what rows and columns $\gamma_{jl}$ can have nonzero entries to be compatible with the previous choices, but we will obtain all possible $\tilde\nu$ in this way.  If we only choose $\tilde\nu$ such that $\chi^{\tilde\nu}(u_{\tilde\mu})\neq 0$, then the nonzero entries of $\gamma_{jl}$ can only be in rows and columns such that 
\begin{enumerate}
\item[(a)] there are no nonzero entries in $\tilde\nu^{(jl)}-\gamma_{jl}$ in that column or row, 
\item[(b)] there is no nonzero entry of $\tilde\mu$ strictly below in that column or strictly to the left in that row.
\end{enumerate}  
That is, we have 
$$\beta_j-\sum_{j\prec_\cP k\prec_\cP l} \mu_{jk}-\sum_{l\prec_\cP m} \lambda_{jm}\quad \text{rows and}\quad \beta_{l}-\sum_{j\prec_\cP k\prec_\cP l} \mu_{kl}-\sum_{i\prec_\cP j} \lambda_{il}\quad\text{columns}$$
to choose from for the nonzero entries of $\gamma_{jl}$.  Let 
$\rowsupp(\gamma_{jl})\subseteq \cQ_j$ be the rows in $\gamma_{jl}$ with nonzero entries, and $\colsupp(\gamma_{jl})\subseteq \cQ_l$ be the columns with nonzero entries. Let $\delta_\lambda^{(jl)}\in \cT_\cP^\beta$ be given by $(\delta_\lambda^{(jl)})_{ik}=\delta_{(i,k),(j,l)}\lambda_{jl}$, and let $\lambda^{(jl)}\in \cT_\cP^{\beta}$ label the orbit containing $e^*_{\tilde\nu^{(jl)}}\in \fkut_\cP^*$.  Then we have recursions,
\begin{align*}
\dim_L(\tilde\nu^{(jl)})+\dim_R(\tilde\nu^{(jl)})&=\dim_L(\tilde\nu^{(jl)}-\gamma_{jl})+\dim_R(\tilde\nu^{(jl)}-\gamma_{jl})+\dim_L(\delta_\lambda^{(jl)})+\dim_R(\delta_\lambda^{(jl)})\\
&\hspace*{.5cm} +\wt_{\cQ_j}^\uparrow(\rowsupp(\gamma_{jl}))+\wt_{\cQ_l}^\downarrow(\colsupp(\gamma_{jl}))\\
\nst_{\tilde\mu}^{\tilde\nu^{(jl)}}&=\nst_{\tilde\mu}^{\tilde\nu^{(jl)}-\gamma_{jl}}+\nst_\mu^{(\delta_\lambda^{(jl)})}+\#\{i'\prec j'\prec k'\prec l'\mid j'\in \cQ_j,(\gamma_{jl})_{i'l'}\tilde\mu_{j'k'}=1\}\\
&\hspace*{.5cm}+\#\{i'\prec j'\prec k'\prec l'\mid j'\notin \cQ_j,k'\in \cQ_l,(\gamma_{jl})_{i'l'}\tilde\mu_{j'k'}=1\}\\
\crs(\tilde\nu^{(jl)})&=\crs(\tilde\nu^{(jl)}-\gamma_{jl})+\crs(\lambda^{(jl)})-\crs(\lambda^{(jl)}-\delta_\lambda^{(jl)}) +\crs(\gamma_{jl})\\
&\hspace*{.5cm}+\wt^\uparrow_{\rowsupp(\tilde\nu^{(jl)})\cap \cQ_j-\rowsupp(\gamma_{jl})}(\rowsupp(\gamma_{jl}))+\wt^\downarrow_{\colsupp(\tilde\nu^{(jl)})\cap \cQ_l-\colsupp(\gamma_{jl})}(\colsupp(\gamma_{jl})).
\end{align*}
Iterating these recursions gives
\begin{align*}
\dim_L(\tilde\nu)+\dim_R(\tilde\nu)&= \dim_L(\lambda)+\dim_R(\lambda)+\sum_{j\prec_\cP l} \wt_{\cQ_j}^\uparrow(\rowsupp(\gamma_{jl}))+\wt_{\cQ_l}^\downarrow(\colsupp(\gamma_{jl}))\\
\nst_{\tilde\mu}^{\tilde\nu}&=\nst_\mu^\lambda+\sum_{j\prec_\cP l}\Big( \#\{i'\prec j'\prec k'\prec l'\mid j'\in \cQ_j,(\gamma_{jl})_{i'l'}\tilde\mu_{j'k'}=1\}\\
&\hspace*{2cm}+\#\{i'\prec j'\prec k'\prec l'\mid j'\notin \cQ_j,k'\in \cQ_l,(\gamma_{jl})_{i'l'}\tilde\mu_{j'k'}=1\}\Big)\\
\crs(\tilde\nu)&= \crs(\lambda)+ \sum_{j\prec_\cP l} \Big(\crs(\gamma_{jl})
+\wt^\uparrow_{\rowsupp(\tilde\nu^{(jl)})\cap \cQ_j-\rowsupp(\gamma_{jl})}(\rowsupp(\gamma_{jl}))\\
&\hspace*{2cm}+\wt^\downarrow_{\colsupp(\tilde\nu^{(jl)})\cap \cQ_l-\colsupp(\gamma_{jl})}(\colsupp(\gamma_{jl}))\Big).
\end{align*}
Let
\begin{align*}
\cR_{jl} &=\cQ_j-\{i'\in \cQ_j\mid (\tilde\mu)_{i'l'}\neq 0, \text{ for some $l'\in \cQ_k$, $k< l$}\}-(\rowsupp(\tilde\nu^{(jl)}-\gamma_{jl})\cap \cQ_j)\\
\cC_{jl} &= \cQ_l-\{k'\in \cQ_l\mid (\tilde\mu)_{j'k'}\neq 0, \text{ for some $j'\in \cQ_k$, $k> j$}\}-(\colsupp(\tilde\nu^{(jl)}-\gamma_{jl})\cap \cQ_l),
\end{align*}
denote the rows and columns in which $\gamma_{jl}$ may have nonzero entries.  Let $\tilde\mu_{jl}=\Res_{\cQ_j\times\cQ_l}(\tilde\mu)$ and $\gamma_{jl}^\circ=\gamma_{jl}-\gamma_{jl}\cap \tilde\mu_{jl}$.  Then since $\crs(\gamma_{jl})=\crs(\gamma_{jl}^\circ)$ for $\delta_{\mu_{jl}}^{\gamma_{jl}}\neq 0$,
\begin{align*}
\chi^\lambda_\beta(u_\mu)&=\frac{q^{\dim_L(\lambda)+\dim_R(\lambda)}}{q^{\crs(\lambda)+\nst_\mu^\lambda}}\prod_{j\prec_\cP l} \sum_{\gamma_{jl}, \tilde\nu^{(jl)}\in \cT_{\cP}^{(1^N)}\atop \chi^{\tilde\nu^{(jl)}}(u_{\tilde\mu})\neq 0} \frac{(q-1)^{\lambda_{jl}}q^{ \wt_{\cR_{jl}}^\uparrow(\rowsupp(\gamma_{jl}))+\wt_{\cC_{jl}}^\downarrow(\colsupp(\gamma_{jl}))}}{(1-q)^{|\gamma_{jl}\cap \tilde\mu_{jl}|}q^{\crs(\gamma_{jl}^\circ)+\nst_{\tilde\mu_{jl}}^{\gamma_{jl}\cap\tilde\mu_{jl}}+|\gamma_{jl}^\circ||\tilde\mu_{jl}|}}\\
&=\frac{q^{\dim_L(\lambda)+\dim_R(\lambda)}}{q^{\crs(\lambda)+\nst_\mu^\lambda}}\prod_{j\prec_\cP l}\chi^{(\lambda_{jl})}_{\loc^\lambda_{\mu}(j,l)}(u_{(\mu_{jl})}),
\end{align*}
where the second equality comes by letting $\nu_\cA=\gamma_{jl}\cap \tilde\mu_{jl}$ and $\nu_\cB=\gamma_{jl}^\circ$, and comparing with (\ref{KeyIdentity}).
\end{proof}

We also easily obtain a number of consequences (some of which no doubt have more direct proofs).

\begin{corollary}
For $\nu\in \cT_\cP^\beta$ with corresponding $e_{\tilde\nu}^*\in \fkut_\cP^*$,
$$|P_\beta e_{\tilde\nu}^* P_\beta|=\frac{q^{\dim_L(\lambda)+\dim_R(\lambda)}}{q^{\crs(\lambda)}} \prod_{j\prec_\cP l} |\GL_{\lambda_{jl}}(\FF_q)| \qbin{\beta_j-\sum_{l\prec_\cP m}\lambda_{jm}}{\lambda_{jl}}\qbin{\beta_l-\sum_{i\prec_\cP j}\lambda_{il}}{\lambda_{jl}}.$$
\end{corollary}

\begin{proof} By (\ref{SupercharacterDefinition}) and Theorem \ref{FullFormula}, 
$$|P_\beta e_{\tilde\nu}^* P_\beta|=\chi_\beta^\nu(1)=\frac{q^{\dim_L(\lambda)+\dim_R(\lambda)}}{q^{\crs(\lambda)}} \prod_{j\prec_\cP l} |\GL_{\lambda_{jl}}(\FF_q)| \qbin{\beta_j-\sum_{l\prec_\cP m}\lambda_{jm}}{\lambda_{jl}}\qbin{\beta_l-\sum_{i\prec_\cP j}\lambda_{il}}{\lambda_{jl}},$$
as desired.
\end{proof}

Recall that the $\CC$-vector space of functions $\mathrm{f}(G)=\{\psi:G\rightarrow \CC\}$ of a group $G$ has a canonical inner product given by
$$\langle \psi, \theta\rangle=\frac{1}{|G|}\sum_{g\in G} \psi(g)\overline{\theta(g)};$$
with respect to this inner product, the irreducible characters are an orthonormal basis for the subspace of class functions.  While supercharacters are still orthogonal, they are generally no longer orthonormal.  In our case,
using a similar argument to \cite{DI},  
$$\langle \chi_\beta^\nu,\chi_\beta^\mu\rangle=\delta_{\nu\mu} \chi_\beta^\nu(1)\quad\text{for $\nu,\mu\in \cT_\cP^\beta$},$$
so we may deduce the following result.

\begin{corollary}
For $\nu,\mu\in \cT_\cP^\beta$,
$$\langle \chi_\beta^\nu,\chi_\beta^\mu\rangle=\delta_{\nu\mu}\frac{q^{\dim_L(\lambda)+\dim_R(\lambda)}}{q^{\crs(\lambda)}} \prod_{j\prec_\cP l} |\GL_{\lambda_{jl}}(\FF_q)| \qbin{\beta_j-\sum_{l\prec_\cP m}\lambda_{jm}}{\lambda_{jl}}\qbin{\beta_l-\sum_{i\prec_\cP j}\lambda_{il}}{\lambda_{jl}}.$$
\end{corollary}

Recall, that if $(\{\text{conjugacy classes}\},\Irr(G))$ is the usual character theory of $G$, then 
$$\left\{\begin{array}{@{}c@{}} \text{Normal sub-}\\ \text{groups of $G$}\end{array}\right\}=\{K_A\mid A\subseteq \Irr(G)\},\quad \text{where}\quad K_A=\bigcap_{\chi\in A} \{g\in G\mid \chi(g)=\chi(1)\}.$$
It is therefore natural to ask which normal subgroups are obtained from a given supercharacter theory; I believe this has been largely unexplored.  However, for the case of $\UT_\beta$, the answer to this question is particularly pleasing.

\begin{corollary} Fix an integer composition $\beta=(\beta_1,\ldots,\beta_\ell)$  and let $\cL_\beta$ the usual total order $1<2<\cdots<\ell$,
$$\{K_A\mid A\subseteq \cT_\beta^{\cL_\beta}\}=\left\{\UT_{(\beta,\cP)}\ \bigg|\  \cP\triangleleft \cL_\beta\right\},\quad \text{where}\quad K_A=\bigcap_{\mu\in A} \{g\in \UT_{(\beta,\cP)}\mid \chi_\beta^\mu(g)=\chi_\beta^\mu(1)\}.$$
\end{corollary}
\begin{proof}
Let $A\subseteq \cT_\beta^{\cL_\beta}$, and let $\cP^A$ be the subposet of $\cL_\beta$ given by
$$j\prec_{\cP^A} k \quad \text{if and only if } \quad i\leq j<k\leq l \text{ implies } \lambda_{il}=0\text{ for all $\lambda\in A$}.$$
To check that $\cP^A$ is a poset it suffices to check transitivity.  Suppose $j\prec_{\cP^A} k\prec_{\cP^A} l$.  Then $\lambda_{im}=0$ for all $i\leq j<k\leq m$ and $i\leq k<l\leq m$, so also for all $i\leq j<l\leq m$.  Thus, $j\prec_{\cP^A} l$.  

Next we prove that $K_A=\UT_{(\beta,\cP^A)}$.  Let $\mu\in \cT_\beta^{\cL_\beta}$.  By Theorem \ref{FullFormula}, 
$$\chi_\beta^\lambda(u_\mu)=\chi_\beta^\lambda(1) \qquad \text{for all $\lambda\in A$,}$$
if and only if $i\leq j<k\leq l$ implies $\mu_{jk}\neq 0$ only if $\lambda_{il}=0$ for all $\lambda\in A$.  Thus, $u_\mu\in K_A$ if and only if $u_\mu\in \UT_{(\beta,\cP^A)}$. 

Let $(\beta,\cP)$ be a unipotent polytope.   Let $M$ be the set of elements maximal in $\sInt(\cL_\beta)-\sInt(\cP)$.  Then let $A_\cP=\{\lambda\}\subseteq \cT_\beta^{\cL_\beta}$ be given by
$$\lambda_{jk}=\left\{\begin{array}{ll} 1 & \text{if $(j,k)\in M$},\\
0 & \text{otherwise.}\end{array}\right.$$
Then $\UT_{(\beta,\cP)}=\UT_{(\beta,\cP^{A_\cP})}$ and by the previous argument $\UT_{(\beta,\cP^{A_\cP})}=K_{A_\cP}$.  Note that while the functions $A\mapsto \cP_A$ and $\cP\mapsto A_\cP$ do not invert one-another the corresponding group maps $K_A\mapsto \UT_{(\beta,\cP_A)}$ and $\UT_{(\beta,\cP)}\mapsto K_{A_\cP}$ do invert one-another.
\end{proof}

\begin{remark}
In the case of $\beta=(1^N)$, this result says that the normal pattern subgroups are exactly the normal subgroups identified by the $B_\cN$-supercharacter theory (perhaps justifying their existence from a slightly different point of view).
\end{remark}

\end{document}